\documentclass[11pt,reqno]{amsart}
\pdfoutput=1
\usepackage[]{amsmath,amssymb,amsfonts,latexsym,amsthm,enumerate}
\usepackage{amssymb,bbm,mathrsfs,tikz}
\usepackage{enumerate,graphicx,paralist}
\usepackage[noadjust]{cite}

\numberwithin{equation}{section}
\usepackage[colorlinks=true, pdfstartview=FitV]{hyperref}
\newtheorem{maintheorem}{Theorem}
\newtheorem{maincoro}[maintheorem]{Corollary}

\newtheorem{theorem}{Theorem}[section]
\newtheorem*{theorem*}{Theorem}
\newtheorem{lemma}[theorem]{Lemma}

\newtheorem{proposition}[theorem]{Proposition}

\newtheorem{corollary}[theorem]{Corollary}

\theoremstyle{definition}{

\newtheorem*{definition*}{Definition}

\newtheorem*{question*}{Question}
\newtheorem*{example*}{Example}
\newtheorem*{examples*}{Examples}
\newtheorem{remark}[theorem]{Remark}
\newtheorem*{remark*}{Remark}

}

\newcommand{\cF}{{\mathcal{F}}}

\newcommand{\cN}{{\mathcal{N}}}

\newcommand{\cX}{\mathcal X}
\newcommand{\Z}{\mathbb{Z}}
\newcommand{\R}{\mathbb{R}}

\newcommand{\N}{\mathbb{N}}
\renewcommand{\P}{\mathbb{P}}

\newcommand{\E}{\mathbb E}
\newcommand{\bP}{\mathbf P}
\newcommand{\bE}{\mathbf E}

\newcommand{\one}{\mathbbm{1}}
\newcommand{\supp}{\operatorname{supp}}
\renewcommand{\epsilon}{\varepsilon}

\author[E. Lubetzky]{Eyal Lubetzky}
\address{Eyal Lubetzky\hfill\break
Courant Institute 
\\ New York University\\
251 Mercer Street\\ New York, NY 10012, USA.}
\email{eyal@courant.nyu.edu}

\author[C. Thornett]{Chris Thornett}
\address{Chris Thornett\hfill\break
Courant Institute 
\\ New York University\\
251 Mercer Street\\ New York, NY 10012, USA.}
\email{thornett@cims.nyu.edu}

\author[O. Zeitouni]{Ofer Zeitouni}
\address{Ofer Zeitouni\hfill\break
Department of Mathematics\\
Weizmann Institute of Science\\
POB 26, 
Rehovot 76100, Israel\\
and
Courant Institute\\
New York University}
\email{ofer.zeitouni@weizmann.ac.il}

\title[Maximum of BBM in a periodic environment]
{Maximum of Branching Brownian motion\\ in a periodic environment}

\begin{document}

\begin{abstract}
We study the maximum of Branching Brownian motion (BBM) with branching rates that vary in space, via  a periodic function of a particle's location. This corresponds to a variant of the F-KPP equation in a periodic medium, extensively studied in the last 15 years,
admitting pulsating fronts as solutions. Recent progress on this PDE due to Hamel, Nolen, Roquejoffre and Ryzhik~('16) implies tightness for the centered maximum of BBM in a periodic environment. Here we establish the convergence in distribution of specific subsequences of this centered maximum, and identify the limiting distribution.
Consequently, we find the asymptotic shift between the solution to the corresponding F-KPP equation with Heavyside initial data and the pulsating wave, thereby answering a question of Hamel~et~al. Analogous results are given for the cases where the Brownian motion is replaced by an Ito diffusion with periodic coefficients, as well as for nearest-neighbor branching random walks.
\end{abstract}

\maketitle

\section{Introduction}

In classical Branching Brownian motion (BBM), initially there is a single particle at the origin, performing standard Brownian motion; a particle is associated with a rate-1 exponential clock which, upon ringing, causes it to be replaced by two new particles at that location, each evolving thereafter independently in the above manner. The location of the rightmost particle in this process after time $t$, denoted $M_t$, has been extensively studied, due in part to its connection to the behavior of extreme values in the Discrete Gaussian Free Field (and other log-correlated fields; see for instance~\cite{Bovier,Zeitouni}), and to the F-KPP equation, proposed almost a century ago by Fisher~\cite{fisher} and by Kolmogorov, Petrovskii and Piskunov~\cite{KPP} to model the spread of an advantageous gene in a population:
\begin{equation}\label{eq:kpp}	\frac{\partial u}{\partial t} = \frac12 \frac{\partial^2 u}{\partial x^2} + F(u) \end{equation}
	with 
\begin{equation}\label{eq:F-dyadic-bbm} F(u) = u^2-u \,,\quad u(0,x)=\one_{\{x\geq 0\}}\,.\end{equation}
As found by McKean~\cite{McKean}, BBM gives a probabilistic representation to~\eqref{eq:kpp} with  initial conditions $u(0,x)=f(x)$ via $u(t,x) = \E\left[ \prod_{v\in \cN_t} f(x+X_t^{(v)})\right]$, where $\cN_t$ is the set of particles at time $t$, and $X_t^{(v)}$ is the location of the particle $v$ at that time. With this interpretation,
$u(t,x)=\P(\min_{v\in \cN_t}(x+X_t^{(v)})>0)=\P(M_t\leq x)$ solves \eqref{eq:kpp}, \eqref{eq:F-dyadic-bbm}.

Bramson~\cite{Bramson78,Bramson83} was then able to use probabilistic methods---which later had a large impact in the study of extremes of logarithmically 
correlated fields---in a sharp analysis of the maximum of BBM: the median of $M_t$ is $m_t = \sqrt 2 t - \frac3{2\sqrt 2}\log t + c + o(1)$ (with its 
logarithmic ``Bramson correction'' term
 differing from the second order term $\frac1{2\sqrt{2}}\log t$  in the maximum of $e^t$ i.i.d.\ Brownian motions), and $M_t-m_t$ converges in law to a random variable $W$, later identified by Lalley and Selke~\cite{LS87} to be a randomly
 shifted Gumbel random variable.

A version of the F-KPP equation studied by H.~Berestycki, H.~Hamel~\cite{BH02}, followed by a series of papers (cf.~\cite{BHR04,BHN05,HR11} to name a few), replaced the function $F(u)$ in~\eqref{eq:F-dyadic-bbm} by a function $F(x,u)$ that is periodic in $x$ and is a KPP-type nonlinearity.
The case 
\begin{equation}\label{eq:F-periodic}
F(x,u) = g(x) (u^2 - u)\quad\mbox{for $g\in C^1(\R)$ strictly positive and 1-periodic}	
\end{equation}
corresponds to BBM in a periodic environment, where the constant branching rate is replaced by a space-dependent rate prescribed by the function $g$. That is, if 
 $b_v$ and $d_v$ are the birth time and death times of the particle $v$, respectively, and we fictitiously extend $X_s^{(v)}$ beyond its death time $d_v$, then 
\begin{equation}
	\label{eq:g-branch-rates}
\P\left(d_v-b_v>t\;\Big|\; b_v, \, \{X_s^{(v)}: b_v\leq s < \infty\} \right)=\exp\left[-\int_{b_v}^{b_v + t} g\big(X_s^{(v)}\big)\,ds\right]\,. 
\end{equation}
Upon dying, the particle then gives birth to two identical particles at its position who then continue independently under Brownian motion.

A recent breakthrough in analyzing the solution to this flavor of the F-KPP equation (and more generally, allowing $F(x,u)=g(x)f(u)$ for any $f\in C^1([0,1])$ of KPP-type) due to Hamel, Nolen, Roquejoffre and Ryzhik~\cite{HNRR} generalized Bramson's results to the case of periodic environments. In particular, for the maximum $M_t$ of BBM with branching rates  given by $g$ as per~\eqref{eq:F-periodic}, the results of~\cite{HNRR} imply that its median is within order $1$ of
\begin{equation}\label{eq:mt} m_t = v^*t-\frac3{2\lambda^*}\log t + O(1)\quad\mbox{ for explicit $v^*(g),\lambda^*(g)>0$}\,,\end{equation} 
and that $M_t-m_t$ is tight. 

Before we state our main result, let us briefly define $v^*$ and $\lambda^*$. As discussed in \cite{HNRR},
the key is a class of eigenproblems. For every $\lambda \in \R$, let $\gamma(\lambda)$ and $\psi(\cdot,\lambda)$ be the principal eigenvalue and positive eigenfunction of the periodic problem
\begin{equation}\label{eigen}
\begin{array}{rcl}
\displaystyle\frac12\psi_{xx}+\lambda\psi_x+\left(\frac12\lambda^2+ g(x)\right)\psi &=& \displaystyle \gamma(\lambda)\psi\,,\\[4pt]
\displaystyle\psi(x+1,\lambda) &=& \displaystyle\psi(x,\lambda)\,,
\end{array}\end{equation}
normalized so that $\int_0^1\psi(x,\lambda)dx=1$. We define 
\begin{equation}\label{eq:lambda_lambda*}
 v^*:=\min_{\lambda>0}\frac{\gamma(\lambda)}{\lambda}\,,\qquad \lambda^* = \operatornamewithlimits{arg\,min}_{\lambda>0} \frac{\gamma(\lambda)}\lambda\,.	
\end{equation}
The existence of $\lambda^*$ will be proved in Section \ref{sec:prelims_ldp}.

Our main result establishes that, in this setting, $\{M_t-m_t\}$ is tight, and we identify sequences $\{t_n\}$ along which $\{M_{t_n}-m_{t_n}\}$ converge in distribution, along with the limits.
 
\begin{maintheorem}\label{mainthm:law}
Let $M_t$ be the maximum at time $t$ of BBM with branching rates given 
by a $C^1(\R)$ strictly
positive and 1-periodic function $g$ as 
in~\eqref{eq:g-branch-rates}, and let $m_t=v^* t -\frac3{2\lambda^*}\log t$. Then there exist 
a random variable $\Theta$ and a positive, continuous, $1$-periodic function $\nu$ such that
\[
\lim_{t\to\infty}\left|\P\Big(M_t\le m_t+y\Big)-\E\left[\exp\big(-\Theta \nu(m_t+y)e^{-\lambda^* y}\big)\right]\right|=0
\]
for all $y\in\R$. In particular, if $\{t_n\}$ is an increasing sequence such that $(m_{t_n})$ has constant fractional part, then $M_{t_n}-m_{t_n}$ converges in distribution.
\end{maintheorem}
\noindent  The random variable $\Theta$ can be described as an appropriate limit, see \eqref{eq-thetadef} below.
The implicitly defined
function  $\nu(\cdot)$, whose existence
is a consequence of Proposition \ref{righttail}, captures the possible
variability of the tail of the law of $M_t$ as function of the initial condition. We do not rule out the possibility that $\nu(\cdot)$ is actually a constant.

Our analysis has the following consequence for the behavior of the solution to the F-KPP equation in a periodic medium:
\begin{maincoro}\label{maincor:pulsating}
	Let $u(t,x)$ be the solution to the F-KPP equation
	\begin{equation}\label{eq:kpp-periodic}
	\begin{array}{rcl}
	\displaystyle \frac{\partial u}{\partial t} &=& \displaystyle\frac12\frac{\partial^2u}{\partial x^2} + g(x)(u^2-u) \,,\quad \\[8pt]
	\displaystyle u(0,x) &=& \displaystyle \one_{\{x\geq 0\}} \,,\quad
	\end{array}
	\end{equation}
	and let $m_t=v^* t -\frac3{2\lambda^*}\log t$. There exists a function $U(z,x)$, satisfying $U(z+1/v^*,x)=U(z,x-1)$, such that
	\[ \lim_{t\to\infty} \sup_{x\in\R} \left| u(t,x) - U\big(m_t/v^*,x\big) \right| = 0\,.\]
\end{maincoro}
\noindent 
The function $U(\cdot,\cdot)$ is given explicitly in
\eqref{eq-Udef}.

Recall that, in the context of the F-KPP equation in a periodic medium such as the one described in the above corollary, a \emph{pulsating wave} is a solution $U_v(t,x)$ to this equation which satisfies $U(t+1/v,x) = U(t,x-1)$. 
It is known that no such solution exists for $v<v^*$
, whereas if $v\geq v^*$ then such a solution exists and is unique up to time-shifts (see~\cite[p.~467]{HNRR} and the references therein for more details). 
It was shown in~\cite[Thm.~1.2]{HNRR} (in a much greater generality than
we can treat by probabilistic methods, and in particular for more general 
nonlinearities and initial conditions) that
the solution $u(t,x)$ to \eqref{eq:kpp-periodic} satisfies $|u(t,x)- U_{v^*}(t- \frac3{2v^*\lambda^*}\log t + \xi(t),x)|\to 0$ uniformly in $x>0$ for some bounded function $\xi(t)$. 

Combining this with Corollary~\ref{maincor:pulsating}, and using that $m_t/v^*\to\infty$ as $t\to\infty$, we 
see that $|U(t,x)-U_{v^*}(t+\zeta(t),x)|\to0$ as $t\to\infty$, uniformly in $x>0$, for some bounded function $\zeta(t)$. Equation \ref{eq-Udef} then gives a representation
of the pulsating wave.

The structure of the paper is the following. Section
\ref{sec:prelims_ldp} discusses some preliminaries from
the theory of Branching Brownian motions and large deviations, such as Many-to-One (Two) lemmas, tilting, 
and barrier estimates; the proof of a technical ballot theorem 
(Lemma \ref{ballot:Sk}) is postponed to Appendix A. Section \ref{sec-righttail},
which is the heart of the paper and technically the most challenging,
uses these to derive the crucial Proposition \ref{righttail}, which gives the exact asymptotics for the right tail of the law of $M_t$.
The  argument uses
a first and second moment method, in the spirit of
standard proofs in the theory of branching random walks,
see e.g. \cite{Aidekon} and \cite{Shi}; Our argument
is closest to \cite{BDZ}. However,
the periodicity of the coefficients forces significant departure from the standard approach, and in particular forces us to work with certain stopping lines in order to recover an independent increments structure. The proof of the main results follows relatively quickly from the tail
estimate, and is provided in Section \ref{sec:mainresults}. Finally, Section \ref{sec:extensions} discusses extensions to Branching Brownian Motion with periodic drift and diffusion coefficients, and to branching integer-valued 
random walks with periodic branching rates.

\section{Preliminaries}\label{sec:prelims_ldp}
We begin with some important preliminary estimates. First, we define some of the notation we will need later on.

In addition to our BBM, it will be helpful to consider a standard Brownian motion $\{B_t\}$ and its measure $\bP^x$ under which $B_0=x$. We will abbreviate $\bP:=\bP^0$, and let $\{\cF^B_t\}_{t\ge0}$ be its natural filtration.

Analogously, it will often be convenient to consider a different probability measure $\P^x$ under which the BBM begins at $x$, rather than zero. Note that this equivalent to replacing the branching rate $g$ with $g(x+\cdot)$ and the process $\{X^{(v)}_t\}$ with $\{x+X^{(v)}_t\}$. Replacing $g$ by $g(x+\cdot)$ in \eqref{eigen} above, one sees that the eigenvalues (and in particular $\lambda^*$ and $v^*$) do not change, whereas $\psi$ is replaced by $\psi(x+\cdot,\lambda)$. Notice that since, for any given $\lambda$, $\psi(\cdot,\lambda)$ is positive, continuous and $1$-periodic, it is bounded above and below by positive constants. When we refer to this fact, we will often simply say ``$\psi$ is bounded'' for brevity.

Recall that $\cN_t$ denotes the set of particles alive at time $t$, and that $X_t^{(v)}$ is the position at time $t$ of a particle $v\in \cN_t$. For every $v\in \cN_t$, we extend the latter definition and let $X_s^{(v)}$ for $s \leq t$ denote the position of the unique ancestor of $v$ that is alive at time $s$ (whenever $v\in\cN_s$ this unique ancestor would be $v$ itself). Similarly, for a particle $v$ which is alive before time $t$, we let $M_t^{(v)}$ be the maximum of $X_t^{(w)}$ amongst those particles $w\in\cN_t$ which are direct descendants of $v$.

Let $U=\bigcup_{t\ge0}\cN_t$ be the set of all particles. Given a $\bP$-almost surely finite $\{B_t\}$-stopping time $\tau$, we let $\tau^{(v)}$ be the analogous stopping time for $\{X^{(v)}_t\}$ and define the \emph{stopping line}
\[L_\tau:=\Big\{(v,t)\in U\times[0,\infty):v\in\cN_t,\tau^{(v)}=t\Big\},\]
and consider $\cN_{\tau}:=\{u\in U:(u,t)\in L_\tau\text{ for some }t\}$. Observe that if $\tau=t$ is constant, then $\cN_\tau=\cN_t$, so our notation is consistent.

Finally, we note that we will regularly use $C$ to denote a universal positive constant which may change from line to line. For any specific constant that is kept fixed throughout, we will 
use a different symbol.

\subsection{Many-to-few lemmas and large deviations estimates}
The Many-to-One Lemma and Many-to-Two Lemma, which can be traced back to the works~\cite{P74,KP76} (see also~\cite{HR17}), will both be essential for our analysis. The stopping line version we state, along with a more general introduction to stopping lines, can be found in ~\cite[\S2.3]{Maillard}.

\begin{lemma}[Many-to-One Lemma]\label{lem:many-to-one}
Let $\tau$ be a $\bP^x$-almost surely finite $\{\cF^B_t\}$-stopping time, and let $f:C[0,\infty)\to\R$ be a measurable function such that $f\big(\{B_s\}_{s\ge0}\big)$ is $\cF^B_{\tau}$-measurable. Then
\[
\E^x\left[ \sum_{v\in\cN_\tau} f\big(\{X_s^{(v)}\}_{s\ge0}\big)\right] = \bE^x\left[e^{\int_0^\tau g(B_s)ds}f\big(\{B_s\}_{s\ge0}\big)\right].
\]
\end{lemma}

In the next lemma, we set $I_s=\{z\in \R: z\leq q(s)\}$
where $q(\cdot)$ is some continuous function (in our application, $q(\cdot)$ will always be an affine function). The statement holds in
greater generality and can be deduced from 
\cite[Lemma 2.3.3]{Maillard} by considering the process
$X^{(v)}_\cdot$ stopped upon exiting the domain $\{(s,z):s\le t,z \in I_s\}$.
\begin{lemma}[Many-to-Two Lemma]\label{lem:many-to-two}
Assume the set-up of Lemma \ref{lem:many-to-one}, and in addition suppose $f$ is of the form
\[
f\big(\{B_s\}_{s\ge0}\big) = h(\tau)\one_{\{ B_s \in I_s \text{ for } s\le\tau\}}.
\]
Define
\[
p(s,x,dy):=\E^x\left[\sum_{v\in\cN_s}\one_{\{X^{(v)}_u\in I_u \text{ for }u<s,\,\tau^{(v)}>s,\,X_u^{(v)}\in dy\}}\right].
\]
Then
\begin{multline}\label{mtt:constant}
\E^x\left[\sum_{{\substack{v,w\in\cN_\tau \\ v\neq w}}} f\big(\{X_s^{(v)}\}_{s\ge0}\big)f\big(\{X_s^{(w)}\}_{s\ge0}\big)\right] \\
= 2h(t)^2\int_0^t \int_{I_s} g(y)\left(\E^y\left[\sum_{v\in\cN_{t-s}}\one_{\{X^{(v)}_u\in I_{s+u} \text{ for }u\le t-s\}}\right]\right)^2p(s,x,dy) \, ds
\end{multline}
if $\tau=t$ is a constant, and
\begin{multline}\label{mtt:hitting}
\E^x\left[\sum_{{\substack{v,w\in\cN_\tau \\ v\neq w}}} f\big(\{X_s^{(v)}\}_{s\ge0}\big)f\big(\{X_s^{(w)}\}_{s\ge0}\big)\right] \\
= 2\int_0^\infty \int_{I_s} g(y)\left(\E^y\left[\sum_{v\in\cN_\tau}h\left(s+\tau^{(v)}\right)\one_{\{X^{(v)}_u\in I_{s+u} \text{ for }u\le \tau^{(v)}\}}\right]\right)^2p(s,x,dy) \, ds
\end{multline}
if $\tau$ is the first hitting time of a closed set $F\subset \R$.
\end{lemma}

After applying these lemmas, we will want to use an exponential change of measure (tilt) that makes the $e^{\int_0^\tau g(B_s)ds}$ term vanish. With $\psi$ as in \eqref{eigen}, define the function
\begin{equation}
    \label{eq-phidef}
\phi(x,\lambda) := \lambda + \frac{\psi_x(x,\lambda)}{\psi(x,\lambda)} = \partial_x\Big(\lambda x + \log\psi(x,\lambda)\Big).
\end{equation}
Notice that $\phi(\cdot,\lambda)$ is $1$-periodic, and it follows from the definition and \eqref{eigen} that
\begin{equation}\label{phi:pde}
\frac12\phi_x + \frac12\phi^2 = \gamma(\lambda) - g(x).
\end{equation}
With $\phi(\cdot,\lambda)$ as in \eqref{eq-phidef}, consider the process $\{Y_t\}_{t\ge0}$  satisfying
\[
dY_t = \phi(Y_t,\lambda)dt + dW_t,\qquad Y_0=x,
\]
where $\{W_t\}_{t\ge0}$ is a standard Wiener process, and
let $\bP_{\lambda}^x$ denote the associated measure.

Since $\phi(\cdot,\lambda)$ is $1$-periodic, it follows that the law of $\{Y_t-Y_0\}_{t\ge0}$ is the same under $\bP_{\lambda}^x$ as under $\bP_{\lambda}^{x+1}$. Moreover, this process provides us with our desired tilt.

\begin{lemma}\label{tilt}
For any $x\in\R$, $\lambda>0$, $t>0$ and any measurable function $F:C[0,t]\to\R$, we have that
\begin{equation}\label{eq:tilt}
\bE_{\lambda}^x\Big[ F\big( \{Y_s\}_{s\le t} \big) \Big] = \bE^x\left[\frac{\psi(B_t,\lambda)}{\psi(x,\lambda)}e^{\lambda (B_t-x) + \int_0^t g(B_s) ds - \gamma(\lambda)t}F\big( \{B_s\}_{s\le t} \big)\right].
\end{equation}
\end{lemma}

\begin{proof}
This will follow immediately from the Girsanov formula if we can show that
\begin{multline}\label{tilt:suff}
\int_0^t\phi(B_s,\lambda)dB_s - \frac12\int_0^t\phi^2(B_s,\lambda)ds \\
= \log\psi(B_t,\lambda)-\log\psi(B_0,\lambda) + \lambda(B_t-B_0) + \int_0^t g(B_s)ds - \gamma(\lambda)t.
\end{multline}
(Note that for fixed $t$, the left hand side is uniformly bounded, and so Novikov's condition \cite[Proposition VIII.1.15]{RY} is trivially satisfied.) To prove \eqref{tilt:suff}, first note that by \eqref{eq-phidef} and  Ito's lemma,
\[
\lambda (B_t-B_0) + \log\psi(B_t,\lambda) - \log\psi(B_0,\lambda) = \int_0^t\phi(B_s,\lambda) dB_s + \frac12\int_0^t\phi_x(B_s,\lambda)ds.
\]
Rearranging the above display, we find
\begin{multline*}
\int_0^t\phi(B_s,\lambda)dB_s - \frac12\int_0^t\phi^2(B_s,\lambda)ds \\
= \log\psi(B_t,\lambda)-\log\psi(B_0,\lambda) + \lambda(B_t-B_0) - \int_0^t\left(\frac12\phi_x(B_s,\lambda) + \frac12\phi^2(B_s,\lambda)\right)\, ds;
\end{multline*}
plugging in \eqref{phi:pde}, we obtain \eqref{tilt:suff} and complete the proof.
\end{proof}

\begin{remark}\label{tilt:stopped}
Since the Radon-Nikodym derivative is a martingale, \eqref{eq:tilt} remains true when $t$ is replaced by a bounded stopping time $\tau$. It hence also holds for any stopping time $\tau$ if there exists a deterministic  $t>0$ such that $F$ is zero on $\tau>t$, since on this event $\tau$ is equal to the bounded stopping time $\tau\wedge t$.
\end{remark}

We are interested in performing this tilt with $\lambda=\lambda^*$ of \eqref{eq:lambda_lambda*}, which we have not yet shown to exist. We do this in the following lemma.

\begin{lemma}\label{gamma-props}
The function $\gamma$ is analytic and convex on $\R$. Moreover, there exists a unique $\lambda^*>0$ such that
\[
\frac{\gamma(\lambda^*)}{\lambda^*}=\min_{\lambda>0}\frac{\gamma(\lambda)}{\lambda},
\]
and this quantity is positive.
\end{lemma}

\begin{proof}
That $\gamma$ is analytic follows from standard analytic perturbation theory, see e.g. ~\cite[Example VII.2.12]{Kato}). Moreover, since $\psi$ is bounded, we may apply Lemma \ref{tilt} to deduce
\begin{equation}\label{exp:conv}
\bE^x\left[\exp\left(\lambda B_t + \int_0^t g(B_s) ds\right)\right] = e^{\lambda x + \gamma(\lambda)t+O(1)}.
\end{equation}
(In fact, the $O(1)$ term is also uniform for $\lambda\in K$ if $K$ is compact, but we do not use this.) If $\alpha$ and $\beta$ are the respective lower and upper bounds of $g$, then \eqref{exp:conv} and the identity $\bE^x[e^{\lambda B_t}] = e^{\lambda x+\lambda^2t/2}$ imply that
\begin{equation}\label{gamma-bounds}
\gamma(\lambda) \in \left[\frac{\lambda^2}2+\alpha,\frac{\lambda^2}2+\beta\right].
\end{equation}
It follows that $\frac{\gamma(\lambda)}{\lambda}\to +\infty$ as both $\lambda\to0^+$ and $\lambda\to+\infty$. Thus, since $\gamma$ is continuous, it follows that
\[
v^* = \min_{\lambda>0}\frac{\gamma(\lambda)}{\lambda}
\]
exists, and applying \eqref{gamma-bounds} again implies that $v^*>0$. Let $\lambda^*$ be a minimizer. Then
\[
0 = \frac{d}{d\lambda}\frac{\gamma(\lambda)}{\lambda}\bigg|_{\lambda=\lambda^*} = \frac{\gamma'(\lambda^*) - \frac{\gamma(\lambda^*)}{\lambda^*}}{\lambda^*},
\]
that is, $v^*=\gamma'(\lambda^*)$. If $\lambda^*$ were not unique, the convexity of $\gamma$ would imply that $\gamma'$ is constant on an interval, and analyticity would then imply that $\gamma$ is affine. But this is clearly impossible by \eqref{gamma-bounds}, so $\lambda^*$ is unique.
\end{proof}

From now on, we will deal exclusively with $\lambda=\lambda^*$. Since $\lambda^*v^*-\gamma(\lambda^*)=0$ and $\psi$ is bounded, Lemma \ref{lem:many-to-one} and Lemma \ref{tilt} imply
\begin{align*}
\E^x\left[\sum_{v\in \cN_t} \one_{\{X^{(v)}_t-v^*t \in [y,y+1]\}} \right] &=\bE^x\left[e^{\int_0^tg(B_s)ds}\,;\,B_t-v^*t\in[y,y+1]\right] \\
&= e^{-\lambda^*(v^*t+y+O(1))+\gamma(\lambda^*)t}\bP^x_{\lambda^*}\Big(Y_t-v^*t \in [y,y+1]\Big) \\
&= e^{-\lambda^*y + O(1)}\bP^x_{\lambda^*}\Big(Y_t-v^*t \in [y,y+1]\Big),
\end{align*}
for any $x,y\in\R$. Thus, studying asymptotics for $\{Y_t\}$ will be crucial in our proof. We begin with a large deviations principle.

\begin{lemma}\label{ldp}
Let $\mu^x_t$ be the law of $\{\frac{Y_t}t\}$ under $\bP^x_{\lambda^*}$. Then $\{\mu^x_t\}$ satisfies a large deviations principle with good, convex rate function
\[
I(z) := \gamma^*(z) - \Big\{\lambda^*z - \gamma(\lambda^*)\Big\},
\]
where $\gamma^*(z) = \sup_{\lambda\in\R}\big\{\lambda z - \gamma(\lambda)\big\}$ denotes the Legendre transform of $\gamma$.
\end{lemma}

\begin{proof}
That $I$ is a convex rate function follows from the convexity 
of $\gamma$ and the definition of $\gamma^*$. For the large deviations principle, observe from \eqref{eq:tilt} that for any $\eta\in\R$,
\begin{equation}\label{ldp:eq1}
\frac1t\log\bE_{\lambda^*}^x\left[e^{\eta Y_t}\right] = \frac1t\log\bE^x\left[\frac{\psi(B_t,\lambda^*)}{\psi(x,\lambda^*)}e^{(\lambda^*+\eta)(B_t-x)+\int_0^tg(B_s)ds}\right]+\frac{\lambda^*x}t-\gamma(\lambda^*).
\end{equation}
Since $\psi$ is bounded, it follows from \eqref{exp:conv} that
\[
\lim_{t\to\infty}\frac1t\log\bE_{\lambda^*}^x\left[e^{\eta Y_t}\right] = \gamma(\lambda^*+\eta) - \gamma(\lambda^*).
\]
The right hand side of the above display is a globally differentiable function of $\eta$ by Lemma \ref{gamma-props}, and so by the G\"artner--Ellis Theorem (cf., e.g.,~\cite[\S2.3]{DZ}), it follows that $\{\mu^x_t\}$ satisfies a large deviations principle with rate function
\[
J(z) := \sup_{\eta\in\R}\Big\{\eta z - \big[\gamma(\lambda^*+\eta) - \gamma(\lambda^*)\big]\Big\} = I(z),
\]
as claimed. Finally, to see that $I$ is a good rate function, fix $\lambda^+>\lambda^*>\lambda^-$. Then, if $z>0$ then
\[ (\lambda^+ -\lambda^*) z- (\gamma(\lambda^+)-\gamma(\lambda^*)) \leq  I(z)\]
while if $z<0$ then
\[ (\lambda^- -\lambda^*) z- (\gamma(\lambda^-)-\gamma(\lambda^*))\leq I(z).\]
This implies that 
$I(z)/|z|\to_{|z|\to \infty} \infty$, as needed.
\end{proof}

\begin{corollary}\label{lln}
For any $x\in\R$, one has that ${Y_t}/t \to v^*$, $\bP_{\lambda^*}^x$-almost surely.
\end{corollary}

\begin{proof}
Note that $I(z)=0$ if and only if $z=\gamma'(\lambda^*)=v^*$. Thus, for any $\epsilon>0$, one has $\delta:=I(v^*-\epsilon/2)\wedge I(v^*+\epsilon/2)>0$, and
\begin{equation}\label{point:ldp}
\bP_{\lambda^*}^x\Big(|Y_t-v^*t|>t\epsilon/2\Big) \le Ce^{-\delta t/2}.
\end{equation}
Moreover, by applying Lemma \ref{tilt} and performing a standard computation for Brownian motion, one has
\begin{equation}\label{path:ldp}
P(x,T):=\bP_{\lambda^*}^x\left(\max_{t\in[0,1]}|Y_t-Y_0-v^*t|>T\varepsilon/2\right) \le Ce^{-\delta T/2}
\end{equation}
uniformly for $x\in[0,1]$. Observing that $P(x,T)$ is $1$-periodic in $x$, the bound in \eqref{path:ldp} is thus uniform for $x\in\R$. Thus, for any $n\in\N$, one has
\begin{align*}
\bP_{\lambda^*}^x\left(\bigcup_{t\in[n,n+1]}\{|Y_t-v^*t|>t\varepsilon\}\right) &\le \bP_{\lambda^*}\Big(|Y_n-v^*n|>n\varepsilon/2\Big) \\
& \qquad\qquad + \bP_{\lambda^*}^x\Big(\max_{t\in[0,1]}|Y_{n+t}-Y_n-v^*t)|>n\varepsilon/2\Big) \\
&\le Ce^{-n\delta} + \E_{\lambda^*}^x\Big[P(Y_n,n)\Big] \\
&\le Ce^{-n\delta},
\end{align*}
where the second line follows from the Markov property. It thus follows from the Borel-Cantelli lemma that
\[
\limsup_{t\to\infty}\left|\frac{Y_t}t-v^*\right| \le \varepsilon
\]
$\bP_{\lambda^*}^x$-almost surely, and the claim follows.
\end{proof} 

Combining Corollary \ref{lln} with the discussion above Lemma \ref{ldp}, we see that the expected number of particles within $o(t)$ of $v^*t$ has order $1$. To improve these estimates, one could develop more precise results for the asymptotic behavior of $\{Y_t\}$; however, we will instead use a renewal approach.

We introduce the hitting times: for $k\in \N$,
\begin{equation}\label{Tdef}
T_k:=\inf\big\{t\ge0 \, : \, Y_t \ge k\big\}.
\end{equation}
This is a $\bP_{\lambda^*}^x$-almost surely finite stopping time with exponential tails, since $\frac{k}t < \frac{v^*}2$ for sufficiently large $t>0$, and so
\[
\bP_{\lambda^*}^x\Big(T_k \ge t\Big) \le \bP^x_{\lambda^*}\Big(Y_t\le k\Big) \le \bP^x_{\lambda^*}\left(\frac{Y_t}t\le\frac{v^*}2\right) = e^{-tI(v^*/2)+o(t)}
\]
by Lemma \ref{ldp}. Moreover, $Y_{T_k}=k$ $\bP_{\lambda^*}^x$-almost surely, provided $x\le k$. The key observation is that, by the strong Markov property, $\{T_k-T_{k-1}\}_{k\ge2}$ is IID under $\bP_{\lambda^*}^x$ for any $x\in[0,1]$, with law equal to that of $T_1$ under $\bP_{\lambda^*}$. Moreover, $T_k\to+\infty$ $\bP_{\lambda^*}^x$-almost surely, and so
\[
\bE_{\lambda^*}^0T_1 = \lim_{k\to\infty}\frac{T_k}{k} = \lim_{k\to\infty}\frac{T_k}{Y_{T_k}} = \frac1{v^*}
\]
$\bP_{\lambda^*}^x$-almost surely by Corollary \ref{lln}. This motivates the definition
\begin{equation}\label{Sdef}
S_k := T_k-\frac{k}{v^*},
\end{equation}
which is a mean zero random walk which satisfies $\E_{\lambda^*}[e^{\eta S_1}]<\infty$ for $|\eta|\le\epsilon$, for some $\epsilon>0$.

Intuitively, for $N\sim v^*t$, $Y_t-v^*t$ should have asymptotics which are not too different from those of $N-v^*T_N=-v^*S_N$. Since the asymptotics of sums of IIDs are well understood, it will be beneficial in many cases to convert required estimates into statements about $\{S_k\}$. This motivates the next subsection, where we will state some barrier estimates both of $\{S_k\}$ and $\{Y_t\}$ which we will need in our proof of Theorem~\ref{mainthm:law}.

\subsection{Barrier estimates}
Our first result in this subsection, Lemma \ref{ballot:Sk}, is a collection of barrier estimates for $\{S_k\}$, and will serve as our primary tool moving forward. The lemma
is a slight generalization of \cite[Lemma 2.2 and Lemma 2.3]{BDZ}, and the proof, based on the latter, is given in Appendix A.

In what follows, we will make use of functions $F$ which satisfy, for some $z\ge0$, $a>0$,
\begin{equation}\label{F:assm}
    \supp(F)=[z,z+a] \quad \text{and $F$ is bounded and  Lipschitz on $(z,z+a)$,}
\end{equation}
where $\supp(F)$ is the support of $F$.

\begin{lemma}\label{ballot:Sk}
Let $\{d_N\}$ be a real sequence which  satisfies $|d_N|\le c_0\log N/N$ for some $c_0>0$; define $S_k^{N)}:=S_k+kd_N$. Then there exists $C>0$ such that
\begin{equation}\label{eq:ballot-Sk-upper}
\bP_{\lambda^*}^x\left(y+ S^{(N)}_N \in [z,z+a], \min_{k\le N}\big(y + S_k^{(N)}\big) \ge 0\right) \le \frac{C(y\vee1)(z\vee1)}{N^{3/2}}
\end{equation}
for all $N\ge1$, $y,z\ge 0$ and $x\in[0,1)$, and
\begin{equation}\label{eq:ballot-Sk-lower}
\bP_{\lambda^*}^x\left(y+ S^{(N)}_N \in [0,a], \min_{k\le N}\big(y +  S_k^{(N)}\big) \ge 0\right) \ge \frac{y\vee1}{CN^{3/2}}
\end{equation}
for all $N\ge1$, $0\le y\le \sqrt{N}$ and $x\in[0,1)$. Moreover, there exists $\beta^*>0$ and an increasing function $V:[0,\infty)\to(0,\infty)$ satisfying $\lim_{y\to\infty}\frac{V(y)}{y} = 1$
such that
\begin{equation}\label{eq:ballot-Sk-limit}
    \lim_{N\to\infty}N^{3/2}\bE_{\lambda^*}^x\left[F\big(y+ S_N^{(N)}\big)\,;\,\min_{k\le N}\big( y + S_k^{(N)}\big) \ge 0\right] = \beta^*V^x(y)\int_z^{z+a} F(w)V(w)dw,
\end{equation}
for all $y\ge0$, $x\in[0,1)$, and $F$ satisfying \eqref{F:assm}, where
\[
V^x(y):=\bE_{\lambda^*}^x\left[V(y+S_1)\,;\,y+S_1 \ge 0\right].
\]
Furthermore, if $h$ is a strictly increasing, concave function satisfying $h(0)=0$ and $h(k)\le c_1\log(k+1)$ for some $c_1>0$, then there exists $\delta_m>0$ satisfying $\lim_{m\to\infty}\delta_m=0$ such that
\begin{multline}\label{eq:ballot-Sk-limsup}
    \limsup_{N\to\infty}\bE_{\lambda^*}^x\left[F\big(y+S_N^{(N)}\big)\,;\,\min_{k\le N}\Big(y+y^{1/10}+S_k^{(N)}+h\big(k\wedge(N-k)\big)\Big)\ge0\right] \\
    \le (1+\delta_{y\wedge z})\beta^*V(y)\int_z^{z+a}F(w)V(w)dw
\end{multline}
for all $y\ge1$, $x\in[0,1)$, and $F$ satisfying \eqref{F:assm}. Finally, there exists $C,N_0>0$ such that
\begin{multline}\label{eq:ballot-Sk-concave}
    \bP_{\lambda^*}^x\left( y + S_j^{(N)} \in [z,z+a],\,\min_{k\le j}\Big(y+S_k^{(N)}+h\big(k\wedge(N-k)\big)\Big)\ge0\right) \\
    \le \frac{Cy\big(z+h(N-j)\big)}{N^{3/2}}
\end{multline}
for all $N\ge N_0$, $N/2 \le j\le N$, and $y,z\ge1$.

The constants in \eqref{eq:ballot-Sk-upper}, \eqref{eq:ballot-Sk-lower} and \eqref{eq:ballot-Sk-concave} depend on $a$ and $c_0$ (and $c_1$ in the latter case), but not on $x$ or the choice of $\{d_N\}$; similarly, the constant $\beta^*$ and the function $V$ depend only on the distribution of $S_1$; and the rates of convergence in \eqref{eq:ballot-Sk-limit} and \eqref{eq:ballot-Sk-limsup} may depend on $y$, $c_0$, and the bounds and Lipschitz constant of $F$ (and $c_1$ in the latter case), but not on $x$ or $\{d_N\}$.
\end{lemma}

While Lemma~\ref{ballot:Sk} contains most of the estimates we need to prove Theorem~\ref{mainthm:law}, we will occasionally need versions of \eqref{eq:ballot-Sk-upper} and \eqref{eq:ballot-Sk-lower} for $\{Y_t\}$. We end this section by providing these results.

\begin{lemma}\label{ballot}
Let $t\mapsto q_t$ be a continuous function such that $q_t\ge\epsilon>0$ and $0<v^*-q_t<c_0\log t/t$ for some $\epsilon,c_0>0$, and define
\[
E_{t,y,z}:=\left\{y-\big(Y_t-q_tt)\in[z,z+1],\;\max_{s\le t}\big(Y_s-q_ts\big) \le y\right\}
\]
Then there exists a constant $C>0$ such that
\begin{equation}\label{eq:ballot-Yt-upper}
\bP_{\lambda^*}^x\Big(E_{t,y,z}\Big) \le \frac{Cy(z\vee1)}{t^{3/2}}
\end{equation}
for all $t\ge1$, $y\ge1$, $z\ge0$, and $x\in[0,1)$, and
\begin{equation}\label{eq:ballot-Yt-lower}
\bP_{\lambda^*}^x\Big(E_{t,y,0}\Big) \ge \frac{y}{Ct^{3/2}}
\end{equation}
for all $t\ge1$, $1\le y\le \sqrt{t}$ and $x\in[0,1)$.
\end{lemma}

\begin{proof}
We begin with \eqref{eq:ballot-Yt-upper}. Fix $\epsilon\in(0,v^*/2)$ and assume $|y-z|\le\epsilon t$, since if not, a basic large deviations estimate (c.f. Lemma~\ref{ldp}) yields a sharper bound. For $0\le j\le \lfloor z\rfloor + 1$ and $0\le\ell\le\lfloor t\rfloor$, define $n(j):=\lfloor q_tt+y-z\rfloor + j$  and
\begin{equation*}
B_{j,\ell}:=\Big\{t-T_{n(j)}\in[\ell,\ell+1), \, T_{n(j)+1}>t\Big\}.
\end{equation*}
By the strong Markov property, one has
\begin{multline}\label{ballot:eq1}
\bP_{\lambda^*}^x\Big(E_{t,y,z}\cap B_{j,\ell}\Big) \le
\bE_{\lambda^*}^x\bigg[ \bP_{\lambda^*}^x\Big(y-\big(Y_t-q_tt\big)\in[z,z+1],T_{n(j)+1}>t\,\big|\, \cF^Y_{T_{n(j)}}\Big) \, ; \\
t-T_{n(j)}\in[\ell,\ell+1), \max_{k\le n(j)}\big(k - q_tT_k\big) \le y\bigg].
\end{multline}
Let $\widehat d_t:=\frac1{q_t}-\frac1{v^*}$, observing that
\[
0<\widehat d_t \le c_1\log t/t\le c_2\log n(j)/n(j)
\]
for constants $c_1,c_2>0$ depending on $\epsilon$ and $c_0$, but not on $j$ (since $n(j) \le (v^*+\epsilon)t$). Thus, defining $\widehat S_k^{t}:= S_k-k\widehat d_t$, \eqref{eq:ballot-Sk-upper} of Lemma~\ref{ballot:Sk} implies
\begin{equation}\label{ballot:eq2}
    \bP_{\lambda^*}^x\Big(p_1+S_{n(j)}^{(t)} \in [p_2,p_2+a],\;\min_{k\le n(j)}\big(p_1+S_{n(j)}\big) \ge 0\Big) \le \frac{C(p_1\vee1)(p_2\vee1)}{n(j)^{3/2}}
\end{equation}
for all $p_1,p_2>0$, where the constant $C>0$ does not depend on $j$.
Moreover, one has $k-q_tT_k=-q_t\widehat S_k^{(t)}$, and
\[
\big\{t-T_{n(j)}\in[\ell,\ell+1)\big\} = \left\{y - \big(n(j)-q_tT_{n(j)}\big) \in (L,L+q_t]\right\},
\]
where $L:=z-j-q_t(\ell+1)+\{q_tt+y-z\}$. Hence, by \eqref{ballot:eq2}, we have
\begin{equation}\label{ballot:eq3}
    \bP_{\lambda^*}^x\left(t-T_{n(j)}\in[\ell,\ell+1), \max_{k\le n(j)}\big(k - q_tT_k\big) \le y\right) \!\le \frac{C(y\vee1)(L\vee1)}{n(j)^{3/2}} \le\! \frac{C(y\vee1)(z\vee1)}{t^{3/2}},
\end{equation}
where the second inequality follows since $n(j)\ge (q_t-\epsilon)t$ and $q_t\to v^*$.

Now notice that, on the event $\big\{t-T_{n(j)}\in[\ell,\ell+1)\big\}$, one has
\begin{equation*}
\bP_{\lambda^*}^x\Big(T_{n(j)+1}>t \, \big| \, \cF^Y_{T_n(j)}\Big) \le \bP_{\lambda^*}\Big(T_1>\ell\Big) \le Ce^{-\delta\ell}
\end{equation*}
and 
\begin{equation*}
    \bP_{\lambda^*}^x\Big(y-\big(Y_t-q_tt\big) \in [z,z+1]\,\big|\,\cF^Y_{T_{n(j)}}\Big) \le \bP_{\lambda^*}\Big(\min_{s\ge0}Y_s\le -(j-1)\Big) \le Ce^{-\delta j}
\end{equation*}
for appropriate $\delta>0$. This last inequality comes from the fact that, if $\tau_{j-1}$ is the first time $\{Y_s\}$ hits $-(j-1)$, then $\tau_{j-1}$ is the sum of $j-1$ IID copies of $\tau_1$, which is finite with probability $\rho<1$ since $Y_t/t\to v^*>0$.

Combining the previous two displays, one deduces
\begin{equation}\label{ballot:eq4}
    \bP_{\lambda^*}^x\left(\Big(y-\big(Y_t-q_tt\big)\in[z,z+1],T_{n(j)+1}>t\,\big|\, \cF^Y_{T_{n(j)}}\Big)\right) \le Ce^{-\delta(j\vee\ell)}\le Ce^{-\delta(j+\ell)/2}
\end{equation}
on the event $\big\{t-T_{n(j)}\in[\ell,\ell+1)\big\}$. Plugging \eqref{ballot:eq4} back into \eqref{ballot:eq1} along with \eqref{ballot:eq3} and summing over $j$ and $\ell$, we conclude
\begin{align*}
    \bP_{\lambda^*}^x\Big(E_{t,y,z}\Big) &\le \sum_{j=0}^{\lfloor z\rfloor + 1\rfloor}\sum_{\ell=0}^{\lfloor t\rfloor}\bP_{\lambda^*}\Big(E_{t,y}\cap B_{j,\ell}\Big) \\
    &\le C\sum_{j=0}^{\lfloor z\rfloor + 1\rfloor}\sum_{\ell=0}^{\lfloor t\rfloor} \frac{e^{-\delta(j+\ell)/2}(y\vee1)(z\vee1)}{t^{3/2}} \\
    &\le \frac{C(y\vee1)(z\vee1)}{t^{3/2}},
\end{align*}
proving \eqref{eq:ballot-Yt-upper}.

We are left with proving \eqref{eq:ballot-Yt-lower}. Let $N:=\lfloor q_tt+y-2\rfloor$, and consider the event
\begin{align*}
G_{t,y}:&=\Big\{t-T_N\in\left[\tfrac1{2q_t},\tfrac1{q_t}\right], \, \max_{k\le N}\big(k - q_tT_k\big) \le y - 1\Big\} \\
&=\left\{\frac{y-1}{q_t} + \widehat S_N^{(t)} \in \{q_tt+y\}+[0,\tfrac32],\,\min_{k\le N}\left(\frac{y-1}{q_t} + \widehat S_k^{(t)}\right) \ge 0 \right\},
\end{align*}
where $\widehat S_k^{(t)}$ is as before. On $G_{t,y}$, for $k\le N$ and $s\in[T_{k-1}, T_k]$ one has
\[
Y_s-q_ts \in k - q_tT_{k-1} = 1 + \big(k-1 - q_tT_{k-1}\big) \le y,
\]
and hence by the strong Markov property,
\begin{equation}\label{ballot:eq5}
    \bP_{\lambda^*}^x\Big(E_{t,y,0}\Big) \ge\bE_{\lambda^*}^x\left[ \bP_{\lambda^*}^x\Big(y-\big(Y_t-q_tt\big) \in [0,1],\,\max_{T_N\le s\le t}\big(Y_s-q_ts\big)\le y\,\big|\,\cF_{T_N}^Y\Big)\,;\,G_{t,y}\right].
\end{equation}
Set $\tilde L:=q_tt+y-N$, which is in $[1,2]$. On the event $\big\{t-T_N\in\left[\frac1{2q_t},\frac1{q_t}\right]\big\}$, we have
\begin{align*}
    \bP_{\lambda^*}^x\Big(y-&\big(Y_t-q_tt\big) \in [0,1],\,\max_{T_N\le s\le t}\big(Y_s-q_ts\big)\le y\,\big|\,\cF_{T_N}^Y\Big) \\
    &\ge\min_{s\in\left[\frac1{2q_t},\frac1{q_t}\right]}\bP_{\lambda^*}\Big(\tilde L -\big(Y_s-q_ts\big) \in [0,1],\,\max_{u\le s}\big(Y_u-q_tu\big) \le \tilde L\Big) \\
    &\ge C\min_{s\in\left[\frac1{2q_t},\frac1{q_t}\right]}\bP\Big(\tilde L -\big(B_s-q_ts\big) \in [0,1],\,\max_{u\le s}\big(B_u-q_tu\big) \le \tilde L\Big),
\end{align*}
where the latter inequality follows by applying Lemma~\ref{tilt} and bounding the Radon-Nikodym derivative from below. Using elementary tools for Brownian motion, this last probability can be bounded below by a constant $C>0$. Thus, plugging into \eqref{ballot:eq5}, we see
\begin{equation}\label{ballot:eq6}
\bP_{\lambda^*}^x\Big(E_{t,y,0}\Big) \ge C\bP_{\lambda^*}^x\Big(G_{t,y}\Big).
\end{equation}
Finally, since $0<\widehat d_t \le c_1\log t/t$ for some $c_1>0$, \eqref{eq:ballot-Sk-lower} of Lemma~\ref{ballot:Sk} implies
\begin{equation*}
\bP_{\lambda^*}^x\Big(G_{t,y}\Big) \ge \frac{Cy}{t^{3/2}},
\end{equation*}
and combined with \eqref{ballot:eq6}, this implies \eqref{eq:ballot-Yt-lower} and completes the proof.
\end{proof}

\section{Estimates on the right tail of BBM}
\label{sec-righttail}
In this section, we state and prove several important estimates of $\P^x\Big(M_t>m_t+y\Big)$. To begin with, we focus on upper and lower bounds, but eventually we will need asymptotics as $t\to\infty$ followed by $y\to\infty$.

Throughout this section, we fix $q_t:=m_t/t$. Note that $q_t$   satisfies the conditions of Lemma~\ref{ballot}.

\subsection{Upper and lower bounds}
The main goal of this subsection is to show that $\bP^x\Big(M_t>m_t+y\Big)/ye^{-\lambda^*y}$ is bounded above and below by positive constants. We begin with the following. In analogy with \eqref{Tdef}, introduce for
$k\in \N$ and $v\in U$,
\begin{equation}\label{Tdefv}
T_k^{(v)}:=\inf\big\{t\ge0 \, : \, X_t^{(v)} \ge k\big\}.
\end{equation}
With this, recall the notation $\cN_{T_N}$, see the discussion in the beginning of Section \ref{sec:prelims_ldp}.
\begin{lemma}\label{upperbd:stop}
Define the event
\[
G_{N,z}:=\bigcup_{v\in\cN_{T_N}}\bigcup_{k\le N}\left\{z+\gamma(\lambda^*)\big(T_k^{(v)}-k/v^*\big)-\frac{3k}{2N}\log N+h\big(k\wedge(N-k)\big) < 0\right\},
\]
where $h(i):=3\log(i\vee1)$. Then there exist $C,\delta>0$ such that for all $N\ge 2$, $z\ge1$ and $x\in[0,1)$,
\begin{equation}\label{upperbd:eq}
\P^x\Big(G_{N,z}\Big) \le Cze^{-z}g_{N,\delta}(z),
\end{equation}
 where
\[
g_{N,\delta}(z)= \exp\left\{-\delta z\left(\frac{z}{N\log N}\wedge1\right)\right\}.
\]
\end{lemma}

\begin{remark}
For $x=0$, this is simply Lemma $2.4$ in [BDZ] for the Branching Random Walk $\{T_k^{(v)}\}_{k\ge1,v\in\cN_{T_k}}$. However, this Branching Random Walk does not satisfy the assumptions of that paper (in particular, one has $\E\left[\sum_{v\in\cN_{T_1}}1\right] = +\infty$), and the uniformity in $x\in[0,1)$ may not be immediately obvious, so we provide a full proof.
\end{remark}

\begin{proof}
Let $W_k^{(N,v)}:=\gamma(\lambda^*)\big(T_k^{(v)}-k/v^*\big)-\frac{3k}{2N}\log N$, and let $W_k^{(N)}$ be its analogue when $T_k^{(v)}$ is replaced by $T_k$. Also set $\Phi_{k,N}(z)=z+h\big(k\wedge(N-k)\big)$. By a simple union bound and the Many-to-One Lemma (Lemma~\ref{lem:many-to-one}), one has
\begin{align}
    \P^x\Big(G_{N,z}\Big)
    &\le \sum_{k=0}^{N-1}\P^x\left(\bigcup_{v\in\cN_{T_k}}\Phi_{k+1,N}(z) + W_{k+1}^{(N,v)} <0, \min_{j\le k}\Big(\Phi_{j,N}(z) + W_j^{(N,v)}\Big) \ge 0\right) \nonumber \\
    \label{upperbd:eq1}&\le \sum_{k=0}^{N-1}\bE^x\left[e^{\int_0^{T_k}g(B_s)ds}\,;\, \Phi_{k+1,N}(z) + W_{k+1}^{(N)} <0, \min_{j\le k}\Big(\Phi_{j,N}(z) + W_j^{(N)}\Big) \ge 0 \right].
\end{align}

Observe that $\{\Phi_{k+1,N}(z) + W_{k+1}^{(N)}\} \subseteq \{T_{k+1}\le F_{N,k}(z)\}$ for some deterministic function $F_{N,k}$, so by Remark~\ref{tilt:stopped} we may apply the change of measure in Lemma~\ref{tilt} to deduce
\begin{align}
\bE^x&\left[e^{\int_0^{T_k}g(B_s)ds}\,;\, \Phi_{k+1,N}(z) + W_{k+1}^{(N)} <0, \min_{j\le k}\Big(\Phi_{j,N}(z) + W_j^{(N)}\Big) \ge 0 \right] \nonumber \\
&\le e^{-\Phi_{k+1,N}(z)+\frac{3(k+1)}{2N}\log N}\bP_{\lambda^*}^x\left(\Phi_{k+1,N}(z) + W_{k+1}^{(N)} <0, \min_{j\le k}\Big(\Phi_{j,N}(z) + W_j^{(N)}\Big) \ge 0\right) \nonumber \\
\label{upperbd:eq2}&\le e^{-\Phi_{k+1,N}(z)+\frac{3(k+1)}{2N}\log N}\bP_{\lambda^*}^x\left(\Phi_{k,N}(z)+W_k^{(N)} \in [0,c_0],\, \min_{j\le k}\Big(\Phi_{j,N}(z) + W_j^{(N)}\Big) \ge 0\right)
\end{align}
for some constant $c_0>0$; this last inequality follows since $T_{k+1}>T_k$ and $\Phi_{k+1,N}(z)-\Phi_{k,N}(z)$ is uniformly bounded. It is straightforward to see that
\begin{equation}\label{upperbd:eq3}
e^{-\Phi_{k+1,N}(z)+\frac{3(k+1)}{2N}\log N} \le (k+1)^{3/2}e^{-\Phi_{k+1,N}(z)} \le \frac{Ce^{-z}(k\vee1)^{3/2}}{\big((k\vee1)\wedge(N-k)\big)^3},
\end{equation}
and so we are left with bounding
\[
\Psi^x_{k,N}(z):=\bP_{\lambda^*}^x\left(\Phi_{k,N}(z)+W_k^{(N)} \in [0,c_0],\, \min_{j\le k}\Big(\Phi_{j,N}(z) + W_j^{(N)}\Big) \ge 0\right).
\]

Notice that $W_k^{(N)}/\gamma(\lambda^*)$ is precisely equal to $S_k^{(N)}$ when $d_N=- \frac{3}{2\gamma(\lambda^*)}\frac{\log N}{N}$, so we may apply \eqref{eq:ballot-Sk-upper} of Lemma~\ref{ballot:Sk} to deduce
\begin{equation*}
\Psi^x_{k,N}(z) \le \frac{C\big(z+h(N-k)\big)}{(k\vee1)^{3/2}} \le \frac{Cz(N-k)}{(k\vee1)^{3/2}}
\end{equation*}
for $N/2 \le k \le N$. Notice this bound remains true if we multiply the right hand side by $g_{N,\delta}(z)$ provided $z\le\sqrt{N\log N}$. If either $z>\sqrt{N\log N}$ or $k< N/2$, we instead use a local central limit theorem to write
\begin{equation*}
\Psi^x_{k,N}(z) \le \frac{Ce^{-\delta' z^2/k\vee1}}{(k\vee1)^{1/2}} \le \frac{Cg_{N,\delta}(z)}{(k\vee1)^{3/2}}
\end{equation*}
for suitable $\delta,\delta'>0$. Combining the preceding two displays with \eqref{upperbd:eq3} and plugging into \eqref{upperbd:eq2}, we find
\begin{multline*}
\bE^x\left[e^{\int_0^{T_k}g(B_s)ds}\,;\, \Phi_{k+1,N}(z) + W_{k+1}^{(N)} <0, \min_{j\le k}\Big(\Phi_{j,N}(z) + W_j^{(N)}\Big) \ge 0 \right] \\
\le \frac{Cze^{-z}g_{N,\delta}(z)}{\big((k\vee1)\wedge(N-k)\big)^{2}}.
\end{multline*}
Plugging into \eqref{upperbd:eq1} and summing over $k$, we deduce \eqref{upperbd:eq} and complete the proof.
\end{proof}

Our main application of Lemma~\ref{upperbd:stop} in this section is the following.

\begin{corollary}\label{upperbd:cor}
There exist $C,\delta>0$ such that 
for all $t\ge2$, $y\ge1$ and $x\in[0,1)$,
\begin{equation}\label{upperbd:cor:eq}
    \P^x\Big(M_t>m_t+y\Big) \le Cye^{-\lambda^*y}g_{t,\delta}(y).
\end{equation}
\end{corollary}

\begin{proof}
Fix $\epsilon\in(0,v^*/2)$. If $y>\epsilon t$, one can use a simple first moment estimate (without barrier) to obtain stronger results, so assume $y\le \epsilon t$. Let $N:=\lfloor m_t + y\rfloor$. Observe that
\[
\{M_t>m_t+y\} \subseteq \bigcup_{v\in\cN_{T_N}}\big\{T_N^{(v)}<t\big\}.
\]
Moreover, $\lambda^*N \ge \lambda^*(y-1)+\gamma(\lambda^*)t - \frac32\log t \ge \lambda^*(y-c_0) + \gamma(\lambda^*)t - \frac32\log N$ for some constant $c_0>0$. The latter of these inequalities follows since $y\le \epsilon t$ and hence $N/t\in[v^*-\epsilon,v^*+\epsilon]$ for sufficiently large $t$. Thus, using $\gamma(\lambda^*)=\lambda^*v^*$, we deduce
\begin{align*}
\big\{T_N^{(v)}<t\big\} &\subseteq \Big\{\gamma(\lambda^*)\big\{T_N^{(v)} - N/v^*\big) < \gamma(\lambda^*)t-\lambda^*N\Big\} \\
&\subseteq \Big\{\gamma(\lambda^*)\big(T_N^{(v)}-N/v^*\big) < -\lambda^*(y-c_0) + \frac32\log N\Big\}.
\end{align*}
Hence,
\[
\big\{M_t>m_t+y\big\} \subseteq G_{N,\lambda^*(y-c_0)},
\]
and the result follows from Lemma~\ref{upperbd:stop}.
\end{proof}

We also require a complementary lower bound, for which we use a second moment method.

\begin{lemma}\label{lem:lowerbd}
There exists a constant $C>0$ such that
for all $t\ge1$, $1\le y\le \sqrt{t}$, and $x\in[0,1]$,
\begin{equation}\label{eq:lowerbd}
\P^x\Big(M_t>m_t + y\Big) \ge Cye^{-\lambda^*y}.
\end{equation}
\end{lemma}

\begin{proof}
Define the random variable
\[
Z_{t,y}:= \sum_{v\in\cN_t}\one_{\{y+2 - (X_t^{(v)}-m_t)\in[0,1],\,\max_{s\le t}(X_s^{(v)}-q_ts)\le y + 2 \}}.
\]
An application of Cauchy-Schwarz yields
\begin{equation}\label{lower:eq1}
\P^x\Big(M_t>m_t + y\Big) \ge \P^x\Big(Z_{t,y}\ge1\Big) \ge \frac{\E^x[Z_{t,y}]^2}{\E^x[Z_{t,y}^2]}.
\end{equation}
Applying the Many-to-One Lemma (Lemma~\ref{lem:many-to-one}), Lemma~\ref{tilt}, and \eqref{eq:ballot-Yt-lower} of Lemma \ref{ballot}, one obtains
\begin{align}
    \E^x\big[Z_{t,y}\big] &= \bE^x\left[e^{\int_0^t g(B_s)ds}\,;\, y+2 - \big(B_t-m_t\big)\in[0,1],\,\max_{s\le t}\big(B_s-q_ts\big)\le y + 2\right] \nonumber \\
    &\ge Ce^{-\lambda^*(m_t+y)+\gamma(\lambda^*)t}\bP^x\left(y+2 - \big(Y_t-m_t\big)\in[0,1],\,\max_{s\le t}\big(Y_s-q_ts\big)\le y + 2\right) \nonumber \\
    \label{lower:eq2} &\ge Cye^{-\lambda^*y}.
\end{align}
(Note that
the boundedness of $\psi$ and $N/t$, and the equality $\lambda^*v^*=\gamma(\lambda^*)$, were also used.) To get an upper bound on $\E^x\big[Z_{t,y}^2\big]$, we first apply the Many-to-Two Lemma (Lemma~\ref{lem:many-to-two}) to write
\begin{multline*}
\E^x[Z_{t,y}^2] - \E^x[Z_{t,y}] 
=2\int_0^t\int_0^\infty g\big(q_ts+y+2-z\big)\\
\qquad \cdot\E^{q_ts+y+2-z}\left[\sum_{v\in\cN_{t-s}}\one_{\{\overline X^{(t,v)}_{t-s}-X_0^{(v)}\in [z-1,z],\,\max_{u\le t-s}\overline X^{(t,v)}_u-X_0^{(v)}\le z\}}\right]^2\ \\
\cdot \E^x\left[\sum_{v\in\cN_s}\one_{\{y+2-\overline X^{(t,v)}_s\in dz,\, \max_{u\le s}\overline X^{(v,t)}_u \le y+2\}}\right]\,ds.
\end{multline*}
For given $s$ and $z$, one can repeat the steps in \eqref{lower:eq2}, replacing the lower bound \eqref{eq:ballot-Yt-lower} with the analogous upper bound \eqref{eq:ballot-Yt-upper}, to obtain
\begin{multline}\label{lower:eq3}
\E^{q_ts+y+2-z}\left[\sum_{v\in\cN_{t-s}}\one_{\{\overline X^{(t,v)}_{t-s}-X_0^{(v)}\in [z-1,z],\,\max_{u\le t-s}\overline X^{(t,v)}_u-X_0^{(v)}\le z\}}\right] \\
\le Ce^{\frac{3(t-s)}{2t}\log t-\lambda^*z}\frac{z\vee1}{(t-s)^{3/2}\vee1}.
\end{multline}
Plugging this into the previous display, splitting the inner region of integration into $\{[k-1,k]:k\in\N\}$, and using the boundedness of $g$, we find
\begin{multline}\label{lower:eq4}
\E^x[Z_{t,y}^2] - \E^x[Z_{t,y}] \le C\int_0^t\sum_{k=1}^\infty e^{\frac{3(t-s)}t\log t - 2\lambda^*k}\frac{k^2}{(t-s)^3} \\
\cdot\E^x\left[\sum_{v\in\cN_s}\one_{\{y+2-\overline X^{(t,v)}_s\in [k-1,k],\, \max_{u\le s}\overline X^{(t,v)}_u \le y+2\}}\right]\,ds.
\end{multline}
Applying the same argument as \eqref{lower:eq3} to this final expectation, we deduce
\begin{align*}
    \E^x[Z_{t,y}^2] - \E^x[Z_{t,y}] &\le C\int_0^t \sum_{k=1}^\infty e^{\frac{3(t-s)}t\log t - 2\lambda^*k}\frac{k^2}{(t-s)^3} e^{\frac{3s}{2t}\log t-\lambda^*(y-z)}\frac{yk}{s^{3/2}\vee1}ds \\ 
    &\le Cye^{-\lambda^*y}\cdot\int_0^t\frac{t^{3/2}}{[s(t-s)]^{3/2}\vee1} ds\cdot\sum_{k=1}^\infty k^3e^{-\lambda^*k} \\
    &\le Cye^{-\lambda^*y}\cdot\int_0^t\frac1{[s\wedge(t-s)]^{3/2}\vee1} ds \\
    &\le Cye^{-\lambda^*y},
\end{align*}
where $\log t/t\le \log s/s$ for $s\le t$ was used in the second step. Combined with \eqref{lower:eq1} and \eqref{lower:eq2}, this implies \eqref{eq:lowerbd} and completes the proof.
\end{proof}

We end this subsection with the following, which is an obvious consequence of Corollary~\ref{upperbd:cor} and Lemma~\ref{lem:lowerbd}, but will nonetheless be useful to state.

\begin{corollary}\label{bounds:cor}
There exist constants $C,\delta>0$ such that for all $u\ge1$, $z \ge -\log u + 1$, and $x\in[0,1)$,
\begin{equation}\label{eq:bounds-cor-upper}
    \P^x\Big(M_u>v^*u + z\Big) \le Cu^{-3/2}(z+\log u)e^{-\lambda^*z}g_{u,\delta}(z),
\end{equation}
 and for all $u\ge1$, $1\le z\le \sqrt{u}$, and $x\in[0,1)$,
\begin{equation}\label{eq:bounds-cor-lower}
    \P^x\Big(M_u>v^* + z\Big) \ge C^{-1} u^{3/2}ze^{-\lambda^*z}.
\end{equation}
 Both \eqref{eq:bounds-cor-upper} and \eqref{eq:bounds-cor-lower} remain true if $v^*u$ is replaced by $q_tu$ for some $t\ge u^2$.
\end{corollary}

\subsection{Exact asymptotics}
In this subsection, we obtain estimates for $\P^x(M_t>m_t+y)$ as $t\to\infty$ followed by $y\to\infty$. We are primarily concerned with the following.

\begin{proposition}\label{righttail}
There exists a positive, continuous, $1$-periodic function $\nu$ such that
\begin{equation}\label{eq:righttail}
    \lim_{y\to\infty}\limsup_{t\to\infty}\sup_{x\in[0,1)}\left|\frac{\P^x\Big(M_t>m_t+y\Big)}{\psi(x,\lambda^*)\nu(m_t+y)ye^{-\lambda^*(y-x)}} - 1\right| = 0.
\end{equation}
\end{proposition}
To prove this, we will first need to compare the probability of $\{M_t>m_t+y\}$ with a suitable expectation, and then use the sharp result \eqref{eq:ballot-Sk-limit} of Lemma~\ref{ballot:Sk} to find the limiting behavior of this expectation.

In what follows, we will consider $u=u(y)$ satisfying
\begin{equation}\label{u:assm}
0\le u \le y \qquad \text{and} \qquad \lim_{y\to\infty}u = \infty,
\end{equation}
and define the following events:
\begin{align*}
    D_{t,y}^{(v)} &:= \left\{ T_N^{(v)}<t,\, \min_{k\le N}\Big(\lambda^*y + W_k^{(N,v)}\Big) \ge 0 \right\}, \\
    E_{t,y}^{(v)} &:= \left\{ T_N^{(v)}<t,\, \min_{s \le T_N^{(v)}}\Big(y+c_0+q_ts-X_s^{(v)}\Big) \ge 0\right\}, \\
    F_{t,y}^{(v)} &:=\left\{T_N^{(v)}<t,\, \min_{k \le N} \Big(\lambda^*y + \log u + h\big(k\wedge(N-k)\big) + W_k^{(N,v)}\Big) \ge 0\right\},
\end{align*}
where $W_k^{(N,v)} = \gamma(\lambda^*)(T_k^{(v)} - k/v^*) - \frac32\log N$ and $h(i)=3\log(i+1)$ as in Lemma~\ref{upperbd:stop}. Note that, for a suitable choice of constant $c_0$, one has
\begin{equation}\label{ascending}
D_{t,y}^{(v)} \subseteq E_{t,y}^{(v)}\subseteq F_{t,y}^{(v)}
\end{equation}
if $y\le\sqrt{t}$ and $u$ is sufficiently large. To see this, observe first by Taylor expansion that
\[
\frac1q_t - \frac1{v^*} = \frac3{2(v^*)^2\lambda^*}\frac{\log t}t +O\left(\frac{\log t}t\right)^2 =\frac3{2\gamma(\lambda^*)}\frac{\log N}{v^*t}+ O\left(\frac1t\right) = \frac3{2\gamma(\lambda^*)}\frac{\log N}N + O\left(\frac1t\right),
\]
and so
\[
\frac{W_k^{(N,v)}}{\gamma(\lambda^*)} = T_k-k\left(\frac1{v^*}+\frac3{2\gamma(\lambda^*)}\frac{\log N}N\right) = T_k-k/q_t + O(1).
\]
Hence, on $D_{t,y}^{(v)}$, for $k\le N-1$ and $s\in[T_k^{(v)},T_{k+1}^{(v)})$, one has
\[
0\le \frac{q_t}{\gamma(\lambda^*)}\Big(\lambda^*y + W_k^{(N,v)}\Big) = y + q_tT_k^{(v)} - k + O(1) \le y + q_ts - X_s^{(v)} + O(1),
\]
and on $E_{t,y}^{(v)}$, for $k\le N$ one has
\begin{multline*}
0\le y + c_0 + q_tT_k^{(v)} - k = \frac{q_t}{\gamma(\lambda^*)}\Big(\lambda^* y + W_k^{(N,v)}\Big) + O(1) \\
\le \frac{q_t}{\gamma(\lambda^*)}\Big(\lambda^*y + \log u + h\big(k\wedge(N-k)\big) + W_k^{(N,v)}\Big).
\end{multline*}

Now, for $v$ alive before time $t$, let $M^{(v)}_t$ be the maximum of $X^{(w)}_t$ over $w\in\cN_t$ which are descendants of $v$. We define
\begin{align*}
    \Lambda_{t,y}&:=\sum_{v\in\cN_{T_N}} \one_{D^{(v)}_{t,y}\cap\{M^{(v)}_t>m_t+y\}}, \\
    \Delta_{t,y}&:=\sum_{v\in\cN_{T_N}}\one_{E^{(v)}_{t,y}\cap\{M^{(v)}_t>m_t+y\}}, \\
    \Xi_{t,y}&:=\sum_{v\in\cN_{T_N}}\one_{F^{(v)}_{t,y}\cap\{M^{(v)}_t>m_t+y\}}.
\end{align*}
Our first step in proving Proposition~\ref{righttail} is the following.

\begin{lemma}\label{asymp:equiv}
    One has
    \begin{equation}\label{eq:asymp:equiv}
        \lim_{y\to\infty}\limsup_{t\to\infty}\sup_{x\in[0,1)}\left|\frac{\P^x\big(M_t>m_t+y\big)}{\E^x[\Lambda_{t,y}]} - 1 \right| = 0.
    \end{equation}
\end{lemma}

In order to prove Lemma \ref{asymp:equiv}, we first show that the expectations of $\Gamma_{t,y}$ and $\Xi_{t,y}$ are asymptotically equivalent. We are then able to use Lemma \ref{upperbd:stop} and a second moment argument to obtain upper and lower bounds, respectively.

\begin{lemma}\label{asymp:upper}
    One has
    \begin{equation}
        \lim_{y\to\infty}\limsup_{t\to\infty}\sup_{x\in[0,1)}\frac{\E^x[\Xi_{t,y}]}{\E^x[\Lambda_{t,y}]} = 1.
    \end{equation}
\end{lemma}

\begin{proof}
By Lemma~\ref{lem:many-to-one}, we have
\begin{align*}
    \E^x\big[\Xi_{t,y}-\Lambda_{t,y}\big] &= \E^x\left[\sum_{v\in\cN_{T_N}}\one_{(F_{t,y}^{(v)}\setminus E_{t,y}^{(v)})\cap\{M_t^{(v)}>m_t+y\}}\right] \\
    &=\bE^x\left[e^{\int_0^{T_N}g(B_s)ds}\rho_{t-T_N}(m_t+y-N) \, ; \, F_{t,y}\setminus D_{t,y}\right]
\end{align*}
where $D_{t,y}$ (respectively, $F_{t,y}$) is the analogue of $D_{t,y}^{(v)}$ (respectively, $F_{t,y}^{(v)}$) when $\{T^{(v)}_k\}$ is replaced by $\{T_k\}$, and
\begin{equation}\label{def:rho}
    \rho_s(z):=\P\Big(M_s>z\Big).
\end{equation}
Applying Lemma~\ref{tilt}, we have
\begin{align}
\E^x\big[\Xi_{t,y}-\Lambda_{t,y}\big] &= \frac{\psi(x,\lambda^*)}{\psi(N,\lambda^*)}\bE_{\lambda^*}^x\left[e^{-\lambda^*(N-x)+\gamma(\lambda^*)T_N}\rho_{t-T_N}(m_t+y-N);F_{t,y}\setminus D_{t,y}\right] \nonumber \\
&=\frac{\psi(x,\lambda^*)}{\psi(0,\lambda^*)}\int_{-\log u}^{\zeta_{t,y}}e^{-\lambda^*(y-x)+z}\rho_{(\zeta_{t,y}-z)/\gamma(\lambda^*)}\big(u+\{m_t+y-u\}\big) \nonumber \\
\label{asymp:upper:eq1} &\qquad\qquad\qquad\qquad\cdot N^{3/2}\bP_{\lambda^*}^x\Big(\lambda^*y+W_N^{(N)}\in dz,\,F_{t,y}\setminus D_{t,y}\Big),
\end{align}
where $\zeta_{t,y}:=\lambda^*\big(u+\{m_t+y-u\}\big)+\frac32\log\frac tN$ and $W_k^{(N)}$ is the analogue of $W_k^{(N,v)}$ when $\{T_k^{(v)}\}$ is replaced by $\{T_k\}$; recall that $W_k^{(N)}= \gamma(\lambda^*)S_k^{(N)}$ when $d_N=-\frac{3}2{\gamma(\lambda^*)}\frac{\log N}N$.

We split the integral on the right hand side of \eqref{asymp:upper:eq1} into the regions $[-\log u, u^{1/3}]$, $[u^{1/3}, u^{2/3}]$, and $[u^{2/3},\zeta_{t,y}]$. For the first and third of these, we ignore the $D_{t,y}$ term altogether and write
\[
N^{3/2}\bP_{\lambda^*}^x\Big(\lambda^*y+W_N^{(N)}\in [z,z+1], \, F_{t,y}\Big) \le Cy(z+\log u)\vee1
\]
by \eqref{eq:ballot-Sk-concave} of Lemma~\ref{ballot:Sk}. Moreover, observe that
\[
u+\{m_t+y-u\} = v^*\left(\frac{\zeta_{t,y}-z}{\gamma(\lambda^*)}\right) + z/\lambda^* + \frac3{2\lambda^*}\log\frac tN,
\]
and hence
\begin{equation}\label{asymp:upper:eq2}
\rho_{(\zeta_{t,y}-z)/\gamma(\lambda^*)}\big(u+\{m_t+y-u\}\big) \le C(z+\tfrac32\log u)(\zeta_{t,y}-z)^{-3/2}e^{-z}g_{u,\delta}(z)
\end{equation}
by \eqref{eq:bounds-cor-upper} of Corollary~\ref{bounds:cor}. Combining these two estimates, we see that the integral of the right hand side of \eqref{asymp:upper:eq1} is bounded by
\begin{equation}\label{asymp:upper:eq3}
Cye^{-\lambda^*y}u^{-1/2}\log u
\end{equation}
in the region $[-\log u,u^{1/3}]$, and by
\begin{equation}\label{asymp:upper:eq4}
Cye^{-\lambda^*y}\sum_{i=\lfloor u^{2/3}\rfloor}^\infty i^2e^{-\delta i} \le Cye^{-\lambda^* y}e^{-\delta u^{2/3}/2}
\end{equation}
in the region $[u^{2/3},\zeta_{t,y}]$.

In the region $[u^{1/3},u^{2/3}]$, we note that the integrand is bounded and Lipschitz (since $\rho$ is a differentiable function of $s$), and so satisfies \eqref{F:assm}; hence, by \eqref{eq:ballot-Sk-limsup} and \eqref{eq:ballot-Sk-limit} of Lemma~\ref{ballot:Sk}, we have
\begin{multline}\label{asymp:upper:eq5}
\int_{u^{1/3}}^{u^{2/3}}e^{-\lambda^*(y-x)+z}\rho_{(\zeta_{t,y}-z)/\gamma(\lambda^*)}\big(u+\{m_t+y-u\}\big)N^{3/2}\bP_{\lambda^*}^x\Big(\lambda^*y+W_N^{(N)}\in dz, \, F_{t,y}\setminus D_{t,y}\Big) \\
\le 2\delta_{u^{1/3}}\int_{u^{1/3}}^{u^{2/3}}e^{-\lambda^*(y-x)+z}\rho_{(\zeta_{t,y}-z)/\gamma(\lambda^*)}\big(u+\{m_t+y-u\}\big)\\
\cdot N^{3/2}\bP_{\lambda^*}^x\Big(\lambda^*y+W_N^{(N)}\in dz, \, D_{t,y}\Big)
\le 2\delta_{u^{1/3}}\E^x\big[\Lambda_{t,y}\big],
\end{multline}
for sufficiently large $t$, where $\delta_z\to0$ as $z\to\infty$. Plugging \eqref{asymp:upper:eq3}--\eqref{asymp:upper:eq5} into \eqref{asymp:upper:eq1}, we deduce
\begin{equation*}
    \E^x\big[ \Xi_{t,y}\big] \le Cu^{-1/4}ye^{-\lambda^*y} + \big(1+2\delta_{u^{1/3}}\big)\E^x\big[\Lambda_{t,y}\big]
\end{equation*}
for $t$ sufficiently large. To complete the proof, it thus suffices to show
\begin{equation}\label{asymp:upper:eq6}
    \E^x\big[\Lambda_{t,y}\big] \ge Cye^{-\lambda^*y}.
\end{equation}
But this is straightforward -- using the same steps that led to \eqref{asymp:upper:eq1}, we have
\begin{multline}\label{asymp:expecation:eq}
\E^x\big[\Lambda_{t,y}\big] = \frac{\psi(x,\lambda^*)}{\psi(0,\lambda^*)}\int_0^{\zeta_{t,y}}e^{-\lambda^*(y-x)+z}\rho_{(\zeta_{t,y}-z)/\gamma(\lambda^*)}\big(u+\{m_t+y-u\}\big) \\
\cdot N^{3/2}\bP_{\lambda^*}^x\Big(\lambda^*y+W_N^{(N)}\in dz, \, D_{t,y}\Big),
\end{multline}
so restricting to the interval $[1,\sqrt{u}]$ and applying \eqref{eq:ballot-Sk-lower} of Lemma~\ref{ballot:Sk} and \eqref{eq:bounds-cor-lower} of Corollary~\ref{bounds:cor} (in analogue to the computation in \eqref{asymp:upper:eq2}), we deduce
\begin{align*}
\E^x\big[\Lambda_{t,y}\big] &\ge Ce^{-\lambda^*y}\sum_{i=1}^{\lfloor\sqrt{u}\rfloor} e^i\cdot(i+\tfrac32\log u)u^{-3/2}e^{-i}\cdot yi \\
&\ge Cye^{-\lambda^*y}u^{-3/2}\sum_{i=1}^{\lfloor \sqrt{u}\rfloor} i^2 \\
&\ge Cye^{-\lambda^*y},
\end{align*}
demonstrating \eqref{asymp:upper:eq6} and completing the proof.
\end{proof}

Since \eqref{ascending} implies that we have $\Lambda_{t,y}\le\Delta_{t,y}\le\Xi_{t,y}$ for sufficiently large $t,y$, Lemma~\ref{asymp:upper} implies that all three expectations are asymptotically equivalent. This allows us to find an upper bound of $\P^x\Big(M_t>m_t+y\Big)$ of the form $\big(1+o(1)\big)\E^x\big[\Lambda_{t,y}\big]$. In order to find a corresponding lower bound, we use a second moment method. The second moment is controlled in the following Lemma.

\begin{lemma}\label{asymp:lower}
One has
\begin{equation}\label{asymp:lowerbd:eq}
    \lim_{y\to\infty}\limsup_{t\to\infty}\sup_{x\in[0,1)}\frac{\E^x\big[\Delta_{t,y}^2\big]}{\E^x\big[\Delta_{t,y}\big]} = 1.
\end{equation}
\end{lemma}

\begin{proof}
Throughout, we abbreviate $\tilde y = y+c_0$. By  the Many-to-Two Lemma \\  (Lemma~\ref{lem:many-to-two}), we have
\begin{align}
    \E^x&\big[\Delta_{t,y}^2-\Delta_{t,y}\big] \nonumber \\
    &=2\int_0^tds\int_{-\infty}^{(q_ts+\tilde y)\wedge N}g(z)\E^x\left[\sum_{v\in\cN_s}\one_{\{T_N^{(v)}>s,X_s^{(v)}\in dz,\max_{w\le s}(X^{(v)}_w-q_tw)\le\tilde y\}}\right] \nonumber \\
    \label{asymp:lowerbd:eq1}\cdot &\left(\E^z\left[\sum_{v\in\cN_{T_N}}\rho_{t-s-T_N^{(v)}}\big(m_t+y-N\big)\one_{\{T_N^{(v)}<t-s,\,\max_{w\le T_N^{(v)}}(X_w^{(v)}-q_t(s+w))\le\tilde y\}}\right]\right)^2.
\end{align}
We now bound this last expectation, which we abbreviate as $\Psi_{t,s}(y,z)$. For $\zeta=q_ts+y-z'$, one applies the Many-to-One Lemma (Lemma~\ref{lem:many-to-one}) followed by Lemma~\ref{tilt} to find
\begin{multline*}
    \Psi_{t,s}(y,z) 
    = \bE^{q_ts+\tilde y-\zeta}_{\lambda^*}\bigg[e^{-\lambda^*(N-Y_0)+\gamma(\lambda^*)T_N}\rho_{t-s-T_N}(m_t+y-N)\,;\\
    \qquad\qquad T_N<t-s,\,\max_{w\le T_N}\big( Y_w - q_tw - Y_0\big) \le \zeta\bigg] \\
   \!\!\!\!\!\!\!\!\!\!\!\!\!\!\!\!\!\!
   \!\!\!\!\!\!\!\!\!\!\!\!\!\!\!\!\!\!
   =\int_0^{u+\{m_t+y-u\}-c_0} e^{-\lambda^*(\zeta-\xi)}\rho_{(u+\{m_t+y-u\}-\xi)/q_t}\big(u+\{m_t+y-u\}\big)\\
  \qquad\cdot \bE^{q_ts+\tilde y-\zeta}\left[e^{\frac{3T_N}{2t}\log t};\zeta - (N-q_tT_N-Y_0)\in d\xi,\,\max_{w\le T_N}\big( Y_w-q_tw-Y_0\big)\le \zeta\right].
\end{multline*}
As in the steps leading to \eqref{asymp:upper:eq2}, we apply \eqref{eq:bounds-cor-upper} of Corollary~\ref{bounds:cor} to find
\begin{multline*}
\rho_{(u+\{m_t+y-u\}-\xi)/q_t}\big(u+\{m_t+y-u\}\big) \\
\le C(u+\{m_t+y-u\}-\xi)^{-3/2}(\xi+\log u)e^{-\lambda^*\xi}g_{u,\delta}(\xi).
\end{multline*}
Plugging this into the above display and applying \eqref{eq:ballot-Sk-upper} of Lemma~\ref{ballot:Sk} (noting, as in the proof of Lemma~\ref{lem:lowerbd}, that $k-q_tT_k=-q_t(S_k^{(t)}-k\widehat d_t)$ with $0\le \widehat d_t\le c_0\log t/t$), we have
\begin{align*}
    \Psi_{t,s}(y,z) &\le Ce^{-\lambda^*\zeta}\sum_{i=0}^{\lfloor u'-c_0\rfloor} (u'-i)^{-3/2}(i+\log u)g_{u,\delta}(i)e^{\frac{3(t-s)}{2t}\log t} \\
    &\qquad\cdot \bP_{\lambda^*}^{q_ts+\tilde y - \zeta}\left(\zeta - (N-q_tT_N-Y_0)\in [i,i+1],\,\max_{k\le N}(q_tT_k-k-Y_0) \le \zeta\right) \\
    &\le  C\zeta e^{-\lambda^*\zeta}\frac{e^{\frac{3(t-s)}{2t}\log t}}{(N-\lfloor q_ts+\tilde y-\zeta\rfloor)^{3/2}}\sum_{i=0}^{\lfloor u'-c_0\rfloor} (u'-i)^{-3/2}(i+1)(i+\log u)g_{u,\delta}(i),
\end{align*}
where $u'=u+\{m_t+y-u\}$. By splitting the sum into $i\le u/2$ and $i>u/2$, one sees that it is bounded by $Cu^{-3/2}$.
Plugging this bound back into \eqref{asymp:lowerbd:eq1}, we have
\begin{align}
    &\E^x\big[\Delta_{t,y}^2-\Delta_{t,y}\big] \nonumber \\
    &\le C\int_0^t\!\int_0^\infty \E^x\left[\sum_{v\in\cN_s}\one_{\{\tilde y - (X_s^{(v)}-q_ts) \in d\zeta,\,\max_{w\le s}(X_w^{(v)}-q_tw)\le\tilde y\}}\right]\nonumber\\
    &\qquad\qquad\qquad\qquad\qquad\qquad \cdot \frac{(\zeta\vee1)^2e^{-2\lambda^*\zeta}e^{3(t-s)\log t/t}}{u^3[(N-\lfloor q_ts+\tilde y-\zeta\rfloor)\vee1]^3}ds \nonumber \\
    \label{asymp:lowerbd:eq2}&\le C\int_0^t\sum_{j=1}^\infty \E^x\left[\sum_{v\in\cN_s}\one_{\{\tilde y - (X_s^{(v)}-q_ts)\in [j-1,j], \, \max_{w\le s}(X_w^{(v)} - q_tw)\le \tilde y\}}\right]
    \frac{j^2e^{-2\lambda^*j}e^{3(t-s)\log t/t}}{u^3[(t-s)\vee1]^3}ds,
\end{align}
where the last line follows since $N-\lfloor q_ts+\tilde y -\zeta\rfloor = q_t(t-s) + \zeta + O(1)$. Now using the same steps as in the analogous calculation in Lemma~\ref{lem:lowerbd}, one has
\begin{equation*}
\E^x\left[\sum_{v\in\cN_s}\one_{\{\tilde y - (X_s^{(v)}-q_ts)\in [j-1,j], \, \max_{w\le s}(X_w^{(v)} - q_tw)\le \tilde y\}}\right] \\
\le \frac{Cyje^{-\lambda^*(y-j)}e^{\frac{3s}{2t}\log t}}{(s\vee1)^{3/2}},
\end{equation*}
and so plugging this back into \eqref{asymp:lowerbd:eq2}, we deduce
\begin{align*}
    \E^x\big[\Delta_{t,y}^2-\Delta_{t,y}\big] &\le Cu^{-3}ye^{-\lambda^*y}\int_0^t\frac{t^{3/2}e^{\frac{3(t-s)}{2t}\log t}}{[s\vee1]^{3/2}[(t-s)\vee1]^3\vee1}\sum_{j=1}^\infty j^3e^{-\lambda^*j} ds\\
    &\le Cu^{-3}ye^{-\lambda^*y}\int_0^t\frac{t^{3/2}}{[s(t-s)]^{3/2}\vee1} ds \\
    &\le Cu^{-3/2}ye^{-\lambda^*y}\int_0^t\frac1{[s\wedge(t-s)]^{3/2}\vee1}ds \le Cu^{-3}ye^{-\lambda^*y}.
\end{align*}
Finally, by \eqref{asymp:upper:eq6}, we have
\[
\E^x\big[\Delta_{t,y}\big] \ge \E^x\big[\Gamma_{t,y}\big] \ge Cye^{-\lambda^*y},
\]
and so it follows that
\[
\E^x\big[\Delta_{t,y}\big] \le \E^x\big[\Delta_{t,y}^2\big] \le \E^x\big[\Delta_{t,y}\big] + Cu^{-3}ye^{-\lambda^*y} \le (1+Cu^{-3})\E^x\big[\Delta_{t,y}\big],
\]
implying \eqref{asymp:lowerbd:eq} and completing the proof
\end{proof}

With Lemmas~\ref{asymp:upper} and \ref{asymp:lower}, we are able to prove Lemma~\ref{asymp:equiv}.

\begin{proof}[\textit{Proof of Lemma~\ref{asymp:equiv}}]
Observe that
\[
\P^x\Big(M_t>m_t+y\Big) \le \E^x\big[\Xi_{t,y}\big] + \P^x\big(G_{N,\lambda^*y+\log u}\big),
\]
where $G_{N,z}$ is as in Lemma~\ref{upperbd:stop}. Applying this Lemma along with \eqref{asymp:upper:eq6}, one has
\[
\P^x\big(G_{N,\lambda^*y+\log u}\big) \le C(y+\log u)u^{-1}e^{-\lambda^*y} \le Cu^{-1}\E^x\big[\Lambda_{t,y}\big],
\]
where we also used $\log u\le u\le y$. Thus, by Lemma~\ref{asymp:upper} we have
\begin{multline}\label{asymp:equiv:eq1}
    \limsup_{y\to\infty}\limsup_{t\to\infty}\sup_{x\in[0,1)}\frac{\P^x\Big(M_t>m_t+y\Big)}{\E^x\big[\Lambda_{t,y}\big]}
    \\
    \le \limsup_{y\to\infty}\limsup_{t\to\infty}\sup_{x\in[0,1)}\left(\frac{\E^x\big[\Xi_{t,y}\big]}{\E^x\big[\Lambda_{t,y}\big]} + Cu^{-1}\right) = 1.
\end{multline}
For the lower bound, we have
\[
\P^x\Big(M_t>m_t+y\Big) \ge \frac{\E^x\big[\Delta_{t,y}\big]^2}{\E^x\big[\Delta_{t,y}^2\big]} \ge \E^x\big[\Lambda_{t,y}\big]\cdot\frac{\E^x\big[\Delta_{t,y}\big]}{\E^x\big[\Delta_{t,y}^2\big]},
\]
and so by Lemma~\ref{asymp:lower} we have
\begin{equation}\label{asymp:equiv:eq2}
\liminf_{y\to\infty}\liminf_{t\to\infty}\inf_{x\in[0,1)}\frac{\P^x\Big(M_t>m_t+y\Big)}{\E^x\big[\Lambda_{t,y}\big]} \ge \liminf_{y\to\infty}\liminf_{t\to\infty}\inf_{x\in[0,1)}\frac{\E^x\big[\Delta_{t,y}\big]}{\E^x\big[\Delta_{t,y}^2\big]} = 1.
\end{equation}
Combining \eqref{asymp:equiv:eq1} and \eqref{asymp:equiv:eq2} yields \eqref{eq:asymp:equiv}, completing the proof.
\end{proof}

With Lemma~\ref{asymp:equiv} at our disposal, we may proceed with the proof of Proposition~\ref{righttail}.

\begin{proof}[\textit{Proof of Proposition~\ref{righttail}}]
Recall from \eqref{asymp:expecation:eq} that
\begin{multline*}
\E^x\big[\Lambda_{t,y}\big] = \frac{\psi(x,\lambda^*)}{\psi(0,\lambda^*)}\int_0^{\zeta_{t,y}}e^{-\lambda^*(y-x)+z}\rho_{(\zeta_{t,y}-z)/\gamma(\lambda^*)}\big(u+\{m_t+y-u\}\big) \\
\cdot N^{3/2}\bP_{\lambda^*}^x\Big(\lambda^*y+W_N^{(N)}\in dz, \, D_{t,y}\Big),
\end{multline*}
where $\zeta_{t,y}=\lambda^*(u+\{m_t+y-u\})+\frac32\log\frac tN$. The idea now is to let $t\to\infty$ followed by $y\to\infty$ in such a way that $\{m_t\}$ and $\{y\}$ are fixed. To that end, for $p,q\in[0,1)$, let
\[
t_0^p:=1,\quad t_{n+1}^p:=\min\big\{t>t_n^p\,:\,\{m_t\}=p\big\}
\]
and
\[
y_k^q:=k+q.
\]
So as to not burden the notation, we will usually suppress the dependence on $n$ and $k$, but we must careful to note that each convergence is uniform in $p$ and $q$.

Observe along these sequences, one has
\[
\xi_{u,p+q}=\lim_{t\to\infty}\zeta_{t,y}=\lambda^*(u+\{p+q-u\})+\frac32\log v^*,
\]
with the convergence being uniform in $p$ and $q$. Moreover,
\[
z\mapsto e^z\rho_{(\xi_{u,p+q}-z)/\gamma(\lambda^*)}\big(u+\{p+q+u\}\big)
\]
is a positive, bounded, Lipschitz function on $[0,\xi_{u,p+q}]$, with upper bound and Lipschitz constant independent of $p$ and $q$, but possibly depending on $u$. Thus, by \eqref{eq:ballot-Sk-limit} of Lemma~\ref{ballot:Sk}, we have
\begin{multline*}
    \lim_{n\to\infty}\E^x\big[\Lambda_{t^p_n,y}\big] = \frac{\psi(x,\lambda^*)}{\psi(0,\lambda^*)}\int_0^{\xi_{u,p+q}}e^{-\lambda^*(y-x)+z}\rho_{(\xi_{u,p+z}-z)/\gamma(\lambda^*)}\big(u+\{p+q-u\}\big) \\
    \cdot \beta^*V^x\left(\frac y{v^*}\right)V\left(\frac z{\gamma(\lambda^*)}\right)dz
\end{multline*}
uniformly in $x,p$, and $q$. In particular, since $V^x(y/v^*)/y\to1/v^*$ as $y\to\infty$ uniformly in $x\in[0,1)$, this implies
\begin{multline*}
\!\!\!\!\!\!\!\!
\frac{E^x\big[\Lambda_{t,y}\big]}{\psi(x,\lambda^*)ye^{-\lambda^*(y-x)}} \sim \frac{\beta^*}{v^*\psi(0,\lambda^*)} \int_0^{\xi_{u,p+q}}V\left(\frac z{\gamma(\lambda^*)}\right)e^z\rho_{(\xi_{u,p+q}-z)/\gamma(\lambda^*)}\big(u+\{p+q-u\}\big)dz \\
:=F(u,p+q)
\end{multline*}
as $t\to\infty$ followed by $y\to\infty$ along the sequences $\{t_n^p\}$, $\{y_k^q\}$, and the convergence is uniform in $x$, $p$ and $q$. Hence, by Lemma \ref{asymp:equiv}, we have
\begin{equation}\label{righttail:eq1}
    \lim_{k\to\infty}\limsup_{n\to\infty}\sup_{x,p,q\in[0,1)}\left|\frac{\P^x\Big(M_{t_n^p}>m_{t_n^p}+y_k^q\Big)}{\psi(x,\lambda^*)y_k^qe^{-\lambda^*(y_k^q-x)}} - F(u,p+q)\right| =0.
\end{equation}
Notice, however, that the first term does not depend on $u$ at all, whereas the only dependence the second term has on $t,y$ is in its dependence on $u$. Thus, by taking different $u,u'$ satisfying \eqref{u:assm}, we see that
\[
\lim_{u,u'\to\infty}\sup_{p,q\in[0,1)}\left|\frac{F(u,p+q)}{F(u',p+q)}-1\right|=0,
\]
and hence there exists $\nu(p+q)$ such that $\lim_{u\to\infty}F(u,p+q)=\nu(p+q)$, uniformly in $p$ and $q$.

To see that 
$\nu$ is continuous and $1$-periodic, observe that $F(u,w)$ is $1$-periodic in $w$ and continuous except possibly at points for which $w-u$ is an integer; for fixed $w_0$, $\epsilon>0$, one can let $u\to\infty$ along points at which $w-u$ is not an integer for any $w\in(w_0-\epsilon,w_0+\epsilon)$, and then since the convergence is uniform one sees that $\nu$ is continuous on $(w_0-\epsilon,w_0+\epsilon)$. The positivity of $\nu$ then follows from Lemma~\ref{lem:lowerbd}, completing the proof.
\end{proof}

\section{Proof of the Main Results}
\label{sec:mainresults}
With Proposition~\ref{righttail} at our disposal, Theorem~\ref{mainthm:law} follows by a standard cutting argument.

\begin{proof}[\textbf{\emph{Proof of Theorem~\ref{mainthm:law}}}]
We will show a slightly more general version, namely that there exists a positive random variable $Z_\infty$ such that
\begin{equation}\label{eq:general}
    \lim_{t\to\infty}\sup_{x,y\in\R}\bigg|\P^x\Big(M_t-x \le m_t+y\Big) - \E^x\Big[\exp\left\{-Z_\infty\nu(m_t+x+y)e^{-\lambda^*y}\right\}\Big]\bigg| = 0,
\end{equation}

where $\nu$ is the function in Proposition~\ref{righttail}.

Let $\Phi^x_s(z):=\P^x\Big(M_s-x>m_s+z\Big)$. 
For $s<t$, we have
\begin{align}
    \P^x\Big(M_t-x \le m_t + y\Big) 
    &=\E^x\left[\prod_{v\in\cN_s}\P^x\Big(M_t^{(v)} \le m_t + x + y\,\big|\,\cF_s\Big) \right] \nonumber \\
    &=\E^x\left[\prod_{v\in\cN_s}\Big(1-\Phi^{X_s^{(v)}}_{t-s}\big(y+H_s^{(t,v)} \big)\Big)\right] \nonumber \\
    \label{general:eq1} &=\E^x\left[\exp\left\{\sum_{v\in\cN_s}\log\Big(1-\Phi^{X_s^{(v)}}_{t-s}\big(y+H_s^{(t,v)} \big)\Big)\right\}\right]
\end{align}
where $H_s^{(t,v)}:=m_t+m_{t-s}-(X_s^{(v)}-X_0^{(v)})$. Define also $H_s^{(v)}:=v^*s-(X_s^{(v)}-X_0^{(v)})=\lim_{t\to\infty}H_s^{(t,v)}$. Consider the event
\begin{equation}
    \label{eq-As}
A_s:=\left\{\min_{v\in\cN_s}H_s^{(v)} \ge \frac1{4\lambda^*}\log s\right\};
\end{equation}
an application of Corollary~\ref{lem:lowerbd} shows that
\[
\epsilon_s:=\sup_{x\in\R}\P^x\big(A_s^c\big) \to 0\quad\text{as }s\to\infty.
\]
Now fix $\delta>0$. By a Taylor expansion, there exists $\delta'>0$ such that
\begin{equation*}
    \left|\frac{\log(1+z)}z-1\right| < \delta
\end{equation*}
whenever $|z|<\delta'$. By Corollary~\ref{upperbd:cor} and Proposition~\ref{righttail}, there exists $z_0$ such that
\begin{equation*}
    \limsup_{s\to\infty}\left|\frac{\Phi_s^x(z)}{\psi(x,\lambda^*)\nu(m_s+x+z)ze^{-\lambda^*z}} - 1\right|\vee \Phi_s^x(z) < \delta\wedge\delta'
\end{equation*}
for all $x\in\R$ and $z\ge z_0$. Plugging the previous two displays into \eqref{general:eq1}, we find that, for all $s \ge e^{4\lambda^*(z_0-y)}$ and each $p\in[0,1)$,
\begin{multline}\label{general:eq2}
    \limsup_{n\to\infty}\P^x\Big(M_{t_n^p}-x \le m_{t_n^p} + y\Big) \\
    \le \E^x\left[\exp\left\{-e^{-\lambda^*y}\nu(p+x+y)(1-\delta)\big(\Theta_s+y\Gamma_s\big)\right\}\,;\,A_s\right] 
    +\epsilon_s
\end{multline}
and
\begin{multline}\label{general:eq3}
    \liminf_{n\to\infty}\P^x\Big(M_{t_n^p}-x \le m_{t_n^p} + y\Big) \\
    \ge \E^x\left[\exp\left\{-e^{-\lambda^*y}\nu(p+x+y)(1+\delta)\big(\Theta_s+y\Gamma_s\big)\right\}\,;\,A_s\right] 
    -\epsilon_s,
\end{multline}
where $\{t_n^p\}$ is defined in the proof of Proposition~\ref{righttail}, and
\begin{equation}\Gamma_s:=\sum_{v\in\cN_s}\psi(X_s^{(v)},\lambda^*)e^{-\lambda^*H_s^{(v)}},\qquad
 \Theta_s:=\sum_{v\in\cN_s}\psi(X_s^{(v)},\lambda^*)H_s^{(v)}e^{-\lambda^*H_s^{(v)}}.\label{eq-thetas}
\end{equation}
Combining \eqref{general:eq2} and \eqref{general:eq3}, and noting also the inequality
\[
\frac{\Gamma_s}{\Theta_s}\one_{A_s} \le \frac{4\lambda^*}{\log s},
\]
we deduce that
\begin{equation*}
    \lim_{n\to\infty}\P^x\Big(M_{t_n^p}-x\le m_{t_n^p} + y\Big)
    =\lim_{s\to\infty}\E^x\left[\exp\left\{-e^{-\lambda^*y}\nu(p+x+y)\Theta_s\one_{A_s}\right\}\right],
\end{equation*}
uniformly in $x$ and $p$, and in particular both limits exist. Since $\nu$ is continuous and $\lim_{y\to+\infty}e^{-\lambda^*y}=0$, $\lim_{y\to-\infty}e^{-\lambda^*y}=+\infty$, one has that $\E^x\big[e^{-\eta \Theta_s\one_{A_s}}\big]$ converges for all $\eta>0$, and hence by a standard argument that 
\begin{equation}
    \label{eq-thetadef}
\Theta_s\one_{A_s}\Rightarrow \Theta
\end{equation}
for some non-negative random variable $\Theta$ (recall that $A_s$ and 
$\Theta_s$  are defined in \eqref{eq-As} and \eqref{eq-thetas}). Thus, we have shown
\begin{equation}\label{general:eq4}
    \lim_{n\to\infty}\P^x\Big(M_{t_n^p}-x\le m_{t_n^p} + y\Big)
    =\E^x\left[\exp\left\{-\Theta\nu(p+x+y)e^{-\lambda^*y}\right\}\right]
\end{equation}
uniformly in $p$ and $x$. Finally, since $\nu$ is continuous, the right hand side of \eqref{general:eq4} is a continuous function of $y$, and thus the convergence is also uniform in $y$. This completes the proof.
\end{proof}

With \eqref{eq:general}, Corollary~\ref{maincor:pulsating} follows easily.

\begin{proof}[\textbf{Proof of Corollary~\ref{maincor:pulsating}}]
Recall that, if $u(t,x)$ solves the F-KPP equation in a periodic medium as given in~\eqref{eq:kpp-periodic} then 
\[ u(t,x)=\P^x\left(\min_{v\in\cN_t} X_t^{(v)} \geq 0\right) = \P^x\left(\max_{v\in\cN_t}\big(-X_t^{(v)}\big) \le 0\right). \]
Since $\{-X_t^{(v)}\}_{t\ge0, v\in\cN_t}$ is also a Branching Brownian Motion (with branching rates given by the function $x\mapsto g(-x)$, which does not change the eigenvalue $\gamma(\lambda)$ in~\eqref{eigen}, and hence does not change $v^*$ or $\lambda^*$), one can apply the general form \eqref{eq:general} or Theorem~\ref{mainthm:law} to deduce
\[
\lim_{t\to\infty}\sup_{x\in\R}\left|u(t,m_t+x) - \E^{m_t+x}\left[\exp\Big\{-\overline \Theta\overline\nu(0)e^{-\lambda^*x}\Big\}\right]\right|=0,
\]
where $\overline \Theta$ and $\overline\nu$ are the analogues of $\Theta$ and $\nu$, respectively, for the Branching Brownian Motion $\{-X_t^{(v)}\}_{t\ge0,v\in\cN_t}$. In particular, using the change of variables $x\to x-m_t$, we find
\[
\lim_{t\to\infty}\sup_{x\in\R}\Big|u(t,x)-U(m_t/v^*,x)\Big| = 0
\]
where
\begin{equation}
    \label{eq-Udef}
U(z,x):=\E^x\left[\exp\Big\{-\overline \Theta\overline\nu(0)e^{-\lambda^*(x-v^*z)}\Big\}\right].
\end{equation}
One can note in the proof of Theorem~\ref{mainthm:law} that $\Theta_s$ has the same law under $\P^x$ as under $\P^{x+1}$, and so the same is true for $\Theta$, and hence also $\overline \Theta$. It follows that $U$ satisfies $U(z+1/v^*,x)=U(z,x-1)$, completing the proof.
\end{proof}

\section{Extensions}
\label{sec:extensions}
This problem can be extended in a number of possible directions. In this section, we briefly explore two of these possibilities: the case where the inhomogeneity is not just in the branching rate, but also in the offspring distribution and the underlying dynamics of each particle; and the analogous discrete version of this latter case, that is, a Branching Random Walk with spatially dependent offspring distribution.

We note that, throughout this section, we will define measure $\P^x$, $\bP^x$, and $\bP_\lambda^x$. Much like in our main text, we will omit the superscript when $x=0$.

\subsection{Spacial dependence on offspring distribution and underlying dynamics}
Theorem~\ref{mainthm:law} can be substantially generalized without dramatically changing the proof. Specifically, while we have considered a periodic branching rate, the offspring distribution (in this case, binary branching) and underlying dynamics (Brownian motion) are not spatially dependent. In this subsection, we consider a generalization of our problem where all of these factors are allowed to depend on space in a periodic manner.

Let $\{X_t\}$ be a process under the probability measure $\bP^x$ with dynamics
\begin{equation}\label{dynamics}
    dX_t = \mu(X_t)dt + \sigma(X_t)dW_t,\qquad X_0=x,
\end{equation}
where $\mu$ and $\sigma$ are $1$-periodic $C^1$ functions, with $\sigma>0$, and $\{W_t\}$ is a standard Brownian motion. Additionally, let $g$ be a positive, $1$-periodic $C^1$ function, and for each $x\in\R$, let $\pi(x)$ be a distribution on $\N$ such that $\pi_0(x)=\pi_1(x)=0$, its first and second moments $\rho(x)$ and $\kappa(x)$ are bounded functions of $x$, and $x\mapsto \pi_n(x)$ is a $1$-periodic $C^1$ function for each $n$.

Under the measure $\P^x$, let
$\{X_t^{(v)}\}_{t\ge0, v\in\cN_t}$ be a Branching process with branching rate $g(x)$, offspring distribution $\pi(x)$, and underlying dynamics given by \eqref{dynamics}. As before, we define
\[
M_t:=\max_{v\in\cN_t}X_t^{(v)}.
\]
The statement of our extension is similar to Theorem $1$, but we will need a new eigenvalue equation to find $v^*$ and $\lambda^*$. Specifically, let $\gamma(\lambda)$ and $\psi(\cdot,\lambda)$ be the respective principal eigenvalue and positive eigenfunction of the equation
\begin{equation}\label{eigen2}
\begin{array}{rcl}
\frac{\sigma^2(x)}2\psi_{xx}+\Big(\mu(x) + \lambda \sigma^2(x)\Big)\psi_x+\left(\lambda\mu(x) + \frac{\lambda^2\sigma^2(x)}2 + \big(\rho(x)-1\big)g(x)\right)\psi &=& \displaystyle \gamma(\lambda)\psi\,,\\[4pt]
\displaystyle\psi(x+1,\lambda) &=& \displaystyle\psi(x,\lambda)\,,
\end{array}\end{equation}
normalized so $\int_0^1\psi(x,\lambda)dx = 1$.

By integrating \eqref{eigen2} over a single period, it is clear that 
\[
\gamma(\lambda) \in \left[\frac{\lambda^2}2 + \alpha - \lambda \|\mu\|_\infty, \frac{\lambda^2}2 + \beta + \lambda\|\mu\|_\infty\right],
\]
where $\alpha$ and $\beta$ are the respective minimum and maximum of the function $x\mapsto \big(\rho(x)-1\big)g(x)$. It thus follows that $\frac{\gamma(\lambda)}{\lambda}\to+\infty$ both as $\lambda\to 0^+$ and as $\lambda\to +\infty$, and hence
\[
v^* = \min_{\lambda>0}\frac{\gamma(\lambda)}{\lambda} = \frac{\gamma(\lambda^*)}{\lambda^*}
\]
exists. (Note that the uniqueness of $\lambda^*$ follows, as in Lemma~\ref{gamma-props}, from the analyticity of $\gamma$.) As in our main text, we then define $m_t=v^*t - \frac3{2\lambda^*}\log t$.

The analogue of Theorem~\ref{mainthm:law} holds in this setting provided $v^*>0$. Explicitly, we have the following.

\begin{theorem}\label{extthm}
Assume $v^*>0$. Then there exist a random variable $\Theta$ and a positive, continuous, $1$-periodic function $\nu$ such that
\[
\lim_{t\to\infty}\left|\P\Big(M_t\le m_t+y\Big)-\E\left[\exp\big(-\Theta \nu(m_t+y)e^{-\lambda^* y}\big)\right]\right|=0
\]
for all $y\in\R$.
\end{theorem}

Let us briefly state how this would be proved. The Many-to-One and Many-to-Two lemmas in this case say
\[
\E^x\left[\sum_{v\in\cN_t}f\big(\{X^{(v)}_s\}_{0\le s\le t}\big)\right] = \bE^x\left[\exp\left\{\int_0^t\big(\rho(X_s)-1\big)g(X_s)ds\right\}f\big(\{X_s^{(v)}\}_{0\le s\le t}\big)\right]
\]
and
\begin{multline*}
    \E^x\left[\sum_{{\substack{v,w\in\cN_t \\ v\neq w}}} \one_{\{X_s^{(v)},X_s^{(w)}\in I_s \text{ for }s\le t\}}\right]
    = \int_0^t\int ds \, \E^x\left[\sum_{v\in\cN_s}\one_{\{X_u^{(v)} \in I_u \text{ for }u\le s,\,X_s^{(v)} \in dy\}}\right] \\
    \cdot \big(\kappa(y)-\rho(y)\big)g(y) \left(\E^y\left[\sum_{v\in\cN_{t-s}}\one_{\{X_u^{(v)}\in I_{s+u} \text{ for }u\le t-s\}}\right]\right)^2,
\end{multline*}
with analogous statments holding on stopping lines.

As we did earlier, we then define
\[
\phi(x,\lambda) := \lambda + \frac{\psi_x(x,\lambda)}{\psi(x,\lambda)}.
\]
We wish to change measure analogously to Lemma~\ref{tilt}, but since $\{X_t\}$ is not necessarily a martingale, we must be slightly careful. First, note that for fixed $\lambda$,
\[
Z_t:=\int_0^t\phi(X_s,\lambda)dX_s - \int_0^t\phi(X_s,\lambda)\mu(X_s)ds
\]
is a martingale, and its quadratic variation is
\[
[Z]_t=\int_0^t \phi^2(X_s,\lambda)\sigma^2(X_s)ds.
\]
Since $\phi$ and $\sigma$ are both bounded, $e^{[Z]_t/2}<\infty$ for all $t>0$, and so $\{Z_t\}$ satisfies Novikov's condition. Thus, its Dol\'eans-Dade exponential
\[
\mathcal E(Z)_t:=\exp\Big(Z_t-\frac12 Z_t\Big)
\]
is a positive martingale. Define the measure $\bP^x_\lambda$ by
\[
\frac{d\bP_{\lambda}^x}{d\bP^x}\Big|_{\cF_t}:=\mathcal E(Z)_t,
\]
where $\{\cF_t\}$ is the natural filtration of $\{X_t\}$. Then by Girsanov's theorem, under $\bP^x_{\lambda}$, $\{X_t\}$ has dynamics
\[
dX_t = \Big(\mu(X_t) + \phi(X_t,\lambda)\sigma(X_t)\Big)dt + \sigma^2(X_t)dW_t,\qquad X_0 = x,
\]
where we also used
\[
[X,Z]_t = \int_0^t\phi(X_s,\lambda)\sigma(X_s)ds.
\]
Moreover, one can show by \eqref{eigen2} and the definition of $\phi$ that
\[
\frac{\sigma^2(x)}2\Big(\phi_x + \phi^2\Big) + \mu(x)\phi = \gamma(\lambda) - \big(\rho(x)-1\big)g(x),
\]
from which it follows, by the same argument as Lemma~\ref{tilt}, that
\[
\mathcal E(Z)_t = \frac{\psi(X_t,\lambda)}{\psi(X_0,\lambda)}\exp\left\{\lambda (X_t-X_0) + \int_0^t \big(\rho(X_s)-1\big)g(X_s)ds - \gamma(\lambda)t\right\}.
\]
Now, one can once again show that, under $\{\bP^x_{\lambda^*}\}$, $\{X_t/t\}$ satisfies a large deviations principle with good, convex rate function
\[
I(z):=\gamma^*(z) - \big\{\lambda^* z - \gamma(\lambda^*) \big\} = \sup_{\eta\in\R}\Big(\big\{\eta z - \gamma(\eta)\big\} - \big\{\lambda^* z - \gamma(\lambda^*)\big\}\Big),
\]
and so in particular, $I(z)=0$ if and only if $z=v^*$. One can then show $X_t/t\to v^*$, $\bP^x$-almost surely. From here, the same stopping time argument will then work, since if $T_k$ is the first time for $\{X_t\}$ to reach $k$, then under $\bP^0_{\lambda^*}$, $\{T_k\}$ is a sum of IIDs with exponential moments in a neighborhood of the origin.

We remark that we believe this argument should be able to be adapted to the case $v^*<0$. In this case, rather than the stopping times $\{T_k\}$, one considers $\{\widetilde T_k\}$, where $\widetilde T_k$ is the first time that $\{X_t\}$ hits $-k$. While many of the preliminary estimates will be the same, the stopping line argument in Section $3.2$ must be changed -- the majority of particles will hit $-k$ {\textit {before}} the maximum does, and so one would need to instead consider $\widetilde T_k^{(s)}$, the first time {\textit{after $s$}} at which $\{X_t\}$ hits $-k$. Of course, this is no longer a sum of IIDs, so after applying the Many-to-Few lemma and the change of measure, one must decompose events based on whether or not $\widetilde T_k^{(s)} = \widetilde T_k$. We leave the details to the reader.

We also note that the $v^*=0$ case is a peculiarity, because now this method will certainly not work, but we have no reason to believe the result will not be true. Indeed, from the outset, we could have chosen to approach the problem by developing barrier estimates on $\{X_t\}$ (in the main text, $\{Y_t\}$) under $\bP_{\lambda^*}^x$ directly. First, one would need a local central limit theorem to estimate
\[
\bP^x_{\lambda^*}\Big(X_t-x-v^*t\in[z,z+a)\Big),
\]
which can be done with mostly functional analytic methods. This is followed by two local central limit theorems with barrier to estimate
\[
\bP^x_{\lambda^*}\Big(y-(X_t-x-v^*t) \in [z,z+a),\,\min_{s\le t}\big(y-(X_s-x-v^*s)\big)\ge0 \Big),
\]
first for $z$ of order $\sqrt{t}$, then for $z$ in a compact set; see \cite{GLL} for a version of these arguments for an additive functional of a finite Markov chain. Note that, since the drift and volatility of $\{X_t\}$ are spatially-dependent, one should expect an additional term in each estimate corresponding to the invariant distribution of the fractional part of $X_t$, which will not be uniform.

We end this subsection by stating the analogue of Corollary~\ref{maincor:pulsating} under this setting.

\begin{corollary}
	Let $u(t,x)$ be the solution to the equation
	\begin{equation}\label{eq:kpp-periodic1}
	\begin{array}{rcl}
	\displaystyle \frac{\partial u}{\partial t} &=& \displaystyle\frac{\sigma^2(x)}2\frac{\partial^2u}{\partial x^2} + \mu(x)\frac{\partial u}{\partial x} +  F(x,u) \,,\quad \\[8pt]
	\displaystyle u(0,x) &=& \displaystyle \one_{\{x\geq 0\}} \,,\quad
	\end{array}
	\end{equation}
	where
	\[
	F(x,u) := g(x)\left(\sum_{k=2}^\infty \nu_k(x)u^k-u\right),
	\]
	and let $m_t=v^* t -\frac3{2\lambda^*}\log t$. There exists a function $U(z,x)$, satisfying $U(z+1/v^*,x)=U(z,x-1)$, such that
	\[ \lim_{t\to\infty} \sup_{x\in\R} \left| u(t,x) - U\big(m_t/v^*,x\big) \right| = 0\,.\]
\end{corollary}

\subsection{Branching Random Walk}
Finally, we consider a discrete analogue of Theorem~\ref{extthm}. Note that the concept of a branching rate here is not necessary -- any step where a particle does not branch is indistinguishable from the particle having exactly one child. We also restrict our attention to the case where the underlying dynamics consist of particles which either do not move, or move one position to the left or right. This is so that, like in the continuous case, a particle must hit $(n-1)L$ before it hits $nL$, allowing us to use our stopping time analysis.

Consider a Markov chain $\{X_n\}$ on $\Z$ under the probability measure $\bP^x$, with initial position $X(0)=x$ and kernel $p$ satisfying, for all $x,y\in\Z$, $p(x+L,y+L)=p(x,y)$ for some integer $L$ and $p(x,y)=0$ if $|x-y|>1$. For each $x\in\R$, let $\pi(x)$ be a distribution on $\N$ such that $\pi_0(x)=0$, its first and second moments $\rho(x)$ and $\kappa(x)$ are bounded functions of $x$ which are greater than $1$, and $x\mapsto \pi_n(x)$ is an $L$-periodic function for each $n$.

We then consider the following Branching Random Walk under the measure $\P^x$: we begin with a single particle at position $x\in\Z$; at generation $n$, we have a collection of particles $V_n$, where $v\in V_n$ has position $X_n^{(v)}$, and then at step $n+1$ gives birth to a random number of particles given by the distribution $\pi(X_n^{(v)})$, and each of these particles then makes a jump according to the kernel $p(X_n^{(v)},\cdot)$, independently of all other particles in generation $V_{n+1}$. As in the case of Branching Brownian Motion, we are interested in the maximum
\[
M_n:=\max_{v\in V_n}X_n^{(v)}
\]
at time $n$. We anticipate the maximum to be located near $c_1n-c_2\log n$ for some constants $c_1,c_2$. To identify the speed and logarithmic correction, we once again require an eigenvalue equation. In this case, let $\cX:=\big\{h\in\R^\Z:h(\cdot+L)=h\big\}$ and consider the operator $Q_\lambda$ on $\cX$ defined by
\[
Q_\lambda h(x):=\bE^x\left[e^{\lambda (X_n-x)}h(X_n)\prod_{i=0}^{n-1}\rho(X_i)\right].
\]
Note that $Q_\lambda^L$ is a strictly positive operator on the finite dimensional space $\cX$, so by the Perron-Frobenius theorem it has a principal eigenvalue $R(\lambda)$ and positive eigenvector $\psi(\cdot,\lambda)$, normalized so that $\sum_{x=1}^L\psi(x,\lambda)=1$. Once again, by \cite{Kato}, these are analytic functions of $\lambda$.

Let $\gamma(\lambda):=\log R(\lambda)$, and \textit{assume}
\[
v^*:=\min_{\lambda>0}\frac{\gamma(\lambda)}{\lambda} = \frac{\gamma(\lambda^*)}{\lambda^*}
\]
exists. Unlike the continuous case, the existence of a minimizer is not necessarily true; indeed, the case with binary branching and $p(x,y)=\frac12\one_{\{|x-y|=1\}}$ has no minimizer. (Roughly speaking, the minimizer will exist provided the branching is sufficiently slow.) The uniqueness of $\lambda^*$ is automatic, however; as in Lemma~\ref{gamma-props}, this follows from the analyticity of $\gamma$.

With $A_n:= -m_n + \Z$, our main result is then the following.
\begin{theorem}\label{discrete}
Assume $v^*>0$. There exists a random variable $\Theta$ and a positive, $L$-periodic function $\nu$ on $\Z$ such that
\begin{equation}\label{conv:brw}
\lim_{n\to\infty}\sup_{y\in A_n}\left|\P\Big(M_n-m_n\le y\Big)-\E\left[\exp\left\{-\Theta \nu(m_n+y)e^{-\lambda^*y}\right\}\right]\right|=0.
\end{equation}
\end{theorem}

The proof follows the same basic outline as that of Theorem~\ref{mainthm:law}. Much like in the previous subsection, the Many-to-Few lemmas and change of measure require more explanation. First, note that the Many-to-One and Many-to-Two lemmas in this context are
\[
\E^x\left[\sum_{v\in V_n}f\big(\{X^{(v)}_i\}_{0\le i\le n}\big)\right] = \bE^x\left[f\big(\{X_i^{(v)}\}_{0\le i\le n}\big)\prod_{i=0}^{n-1}\rho(X_i)\right]
\]
and
\begin{multline*}
    \E^x\left[\sum_{{\substack{v,w\in V_n \\ v\neq w}}} \one_{\{X_j^{(v)},X_j^{(w)}\in I_j \text{ for }j\le n\}}\right]
    = \sum_{k=0}^{n-1}\sum_{y\in\Z} \E^x\left[\sum_{v\in V_k}\one_{\{X_j^{(v)} \in I_j \text{ for }j< k,\,X_{k-1}^{(v)}=y\}}\right] \\
    \cdot \big(\kappa(y)-\rho(y)\big)\left(\sum_{z\in I_k}p(y,z) \E^y\left[\sum_{v\in V_{n-k}}\one_{\{X_j^{(v)}\in I_{j+k} \text{ for }j\le n-k\}}\right]\right)^2,
\end{multline*}
with analogous statements holding on appropriate stopping lines. Now note that, by definition of $\gamma$ and $\psi$,
\[
p_\lambda(x,y):=\frac{\psi(y,\lambda)}{\psi(x,\lambda)}\rho(x)e^{\lambda (y-x) - \gamma(\lambda)}
\]
is an $L$-periodic kernel on $\Z$ which is positive precisely when $p$ is; if $\bP_\lambda^x$ is the probability measure corresponding to the Markov chain $\{X_n\}$ starting at $x$ and having kernel $p_\lambda$, then
\[
\frac{d\bP_\lambda^x}{d\bP^x}\Big|_{\cF_n} = \frac{\psi(X_n,\lambda)}{\psi(x,\lambda)}\exp\Big\{\lambda(X_n-x) - n\gamma(\lambda)\Big\}\prod_{i=0}^{n-1}\rho(X_i),
\]
where $\cF_n$ is the natural filtration of $\{X_n\}$. Once again, under $\bP^x_{\lambda^*}$, $\{X_n/n\}$ satisfies a large deviations principle with good, convex rate function
\[
I(z):=\gamma^*(z) - \big\{\lambda^* z - \gamma(\lambda^*) \big\} = \sup_{\eta\in\R}\Big(\big\{\eta z - \gamma(\eta)\big\} - \big\{\lambda^* z - \gamma(\lambda^*)\big\}\Big),
\]
which is zero if and only if $z=v^*$. This immediately implies (by the Borel-Cantelli lemma) that $X_n/n\to v^*$, $\bP_{\lambda^*}^x$-almost surely. We then define the stopping times
\[
T_k:=\min\Big\{n\in\N:X_n = kL\Big\},
\]
observing that under $\bP_{\lambda^*}^x$, $\{T_k-T_{k-1}\}_{k\ge2}$ are IID with distribution equal to that of $T_1$ under $\bP_{\lambda^*}$, which has mean $L/v^*$ and possesses finite exponential moments in a neighborhood of the origin.

From here, the proof proceeds much as it does for Theorem~\ref{mainthm:law} or Theorem~\ref{extthm}. However, there is one more important difference: the random walk $S_k:=T_k-kL/v^*$ is lattice. Because of this, we end up with a version of Proposition~\ref{righttail} where $y$ is replaced by $y_n$, the unique element of $(y-1,y]\cap A_n$. To prove this, one must replace $k/v^*$ by $kL/v^*$ in the the definition of $W_k^{(N,v)}$, and then in the definition of $D_{t,y}^{(v)}$ replace $y$ by $y'$, with $y-y'$ uniformly bounded, such that $\lambda^*y' + W_N^{(N,v)}$ is integer valued. This allows us to apply the lattice version of Lemma~\ref{ballot:Sk} when estimating $E^x\big[ \Lambda_{t,y}\big]$ -- see \cite[Section $5$]{BDZ} for details.

We finally note that, as in the continuous case, we could have approached this problem directly, rather than using stopping times. We believe this should give superior results, enabling us to handle any $v^*$n, and also allowing us to handle the case where $p$ is not restricted to nearest neighbor jumps. Since one can always write $\{X_n\}$ as an additive functional of a finite Markov chain, the estimates in \cite{GLL} may be used instead of Lemma~\ref{ballot:Sk}, at least if $\{X_n\}$ is non-lattice. We leave this to future work.

\appendix\section{}
In this short appendix, we provide the proof of Lemma~\ref{ballot:Sk}. 
When $x=0$ and $F=\one_{[z,z+a]}$, the lemma
is almost precisely Lemmas $2.2$ and $2.3$ of \cite{BDZ}, with the following minor modifications.
\begin{enumerate}
    \item In \cite{BDZ}, the statements were not in terms of the measure $\bP_{\lambda^*}$ but instead for the measure $\bP_{\lambda_N}$ under which $\bE_{\lambda_N}[S_1^{(N)}]=0$. However, the proof first shows the statements for $\bP_{\lambda^*}$, and then shows that the Radon-Nikodym derivative between the two converges to $1$ as $N\to\infty$.
    \item In \cite{BDZ}, the analogue of \eqref{eq:ballot-Sk-limit} is written less explicitly as a function of $y,z,$ and $a$. However, examining the proof reveals the form of this function as stated.
    \item Because \cite{BDZ} deals also with
    increments whose law does not possess a density, restrictions on the sign of the sequence $d_N$ are imposed there (see the proof of
    \cite[Corollary A.3]{BDZ}, which is a key step in the proof of the lemmas there, for
    where the sign restriction is used explicitly). This is irrelevant in our setup, and hence this restriction is omitted. 
\end{enumerate}

We now proceed to the proof of Lemma~\ref{ballot:Sk},
for $x\neq 0$ or general $F$. Heuristically, when $\{Y_t\}$ reaches $1$ for the first time, it ``resets" and we may apply the result in its original version, with $N-1$ replacing $N$. Given we assume $F$ is quite regular, its appearance does not add much complication to the proofs.

By the strong Markov property, one has
\begin{multline}\label{ballot-Sk:eq1}
    \bE_{\lambda^*}^x\left[y+S_N^{(N)}\in[z,z+a]\,;\,\min_{k\le N}\big(y+S_k^{(N)}\big)\ge0\right] \\
    = \int_0^\infty\bE_{\lambda^*}\left[w+S_{N-1}^{(N)}\in[z,z+a]\,;\,\min_{k\le N-1}\big(w+S_k^{(N)}\big) \ge 0\right]\bP_{\lambda^*}^x\Big(y+S_1^{(N)}\in dw\Big).
\end{multline}
Along with the $x=0$ case and the inequalities
\[
\bE_{\lambda^*}^x\Big[y+S_1^{(N)}\,;\,y+S_1^{(N)} \ge 0\Big] \le y+\bE_{\lambda^*}^x\left|S_1^{(N)}\right| \le y + C
\]
and
\[
\bE_{\lambda^*}^x\Big[y+S_1^{(N)}\,;\,y+S_1^{(N)} \ge 0\Big] \ge y + \bE_{\lambda^*}^x\Big[S_1^{(N)}\Big] \ge y+C,
\]
this implies \eqref{eq:ballot-Sk-upper} and \eqref{eq:ballot-Sk-lower} (taking $z=0$).

Similarly, multiplying \eqref{ballot-Sk:eq1} by $N^{3/2}$, letting $N\to\infty$, and using the bounded convergence theorem, we obtain \eqref{eq:ballot-Sk-limit} for $F=\one_{[z,z+a]}$. To prove this for $F$ satisfying \eqref{F:assm}, observe that for $K\ge1$, if we set $z_{i,K}:=z+ia/K$,
\begin{multline*}
\Bigg|\bE_{\lambda^*}^x\left[F\big(y+ S_N^{(N)}\big)\,;\,\min_{k\le N}\big( y + S_k^{(N)}\big) \ge 0\right]\\
- \sum_{i=1}^KF(z_{i,K})\bE_{\lambda^*}^x\left[y+S_N^{(N)}\in[z_{i,K},z_{i+1,K}]\,;\,\min_{k\le N}\big(y+S_k^{(N)}\big)\ge0\right]\bigg| \le \frac{L}K,
\end{multline*}
where $L$ is the Lipschitz constant of $F$. Multiplying by $N^{3/2}$ and letting $N\to\infty$, we see
\begin{multline*}
    \limsup_{N\to\infty}\Bigg|N^{3/2}\bE_{\lambda^*}^x\left[F\big(y+ S_N^{(N)}\big)\,;\,\min_{k\le N}\big( y + S_k^{(N)}\big) \ge 0\right]\\
    - \beta^*V^x(y)\sum_{i=1}^KF(z_{i,K})\int_{z_{i,K}}^{z_{i+1,K}}V(w)dw\bigg| \le \frac{L}K.
\end{multline*}
However, we also have
\[
\left|\beta^*V^x(y)\sum_{i=1}^KF(z_{i,K})\int_{z_{i,K}}^{z_{i+1,K}}V(w)dw - \beta^*V^x(y)\int_z^{z+a}F(w)V(w)dw\right| \le \frac{L}{K},
\]
and hence
\begin{multline*}
    \limsup_{N\to\infty}\Bigg|N^{3/2}\bE_{\lambda^*}^x\left[F\big(y+ S_N^{(N)}\big)\,;\,\min_{k\le N}\big( y + S_k^{(N)}\big) \ge 0\right]\\
    - \beta^*V(y)\int_z^{z+a}F(w)V(w)dw\bigg| \le \frac{2L}{K}
\end{multline*}
for any $K\ge1$, which implies \eqref{eq:ballot-Sk-limit}.

Repeating the same procedure as above for \eqref{eq:ballot-Sk-limsup} allows us to prove the case $x=0$, $F$ satisfying \eqref{F:assm} from the case $x=0$, $F=\one_{[z,z+a]}$. For general $x$, we use the strong Markov property to write the left hand side of \eqref{eq:ballot-Sk-limsup} as
\begin{multline}\label{ballot-Sk:eq2}
    \int_0^\infty \bE_{\lambda^*}\bigg[F\big(w-y^{1/10}+S_{N-1}^{(N)}\big)\,;\,
    \min_{k\le N-1}\Big(w+S_k^{(N)}+h\big((k+1)\wedge(N-k-1)\big)\Big)\ge0\bigg] \\
    \cdot\bP_{\lambda^*}^x\Big(y+y^{1/10}+S_1^{(N)}\in dw\Big).
\end{multline}
(Note that the integral should really begin from $-h(1)$, not zero, but since $T_1>0$, $\sup_{x\in[0,1)}\bP_{\lambda^*}^x\Big(y+y^{1/10}+S_1^{(N)} \le 0\Big)=0$ for sufficiently large $y$, and thus the integral over this region vanishes.)

Since $h$ is concave and increasing, there exists $M\ge h(1)>0$ such that $h(k+1)-h(k)\le M$ for all $k\ge0$. Letting $c_0:=1+\frac{M}{h(1)}$, one can show that $h(k+1) \le c_0h(k)$ for all $k\ge1$, and hence
\[
h\big((k+1)\wedge(N-k-1)\big) \le h(1) + c_0h\big(k\wedge(N-1-k)\big)
\]
for all $k\le N-1$. Thus, we may bound \eqref{ballot-Sk:eq2} above by
\begin{multline*}
    \int_0^\infty \bE_{\lambda^*}\bigg[F\big(w-y^{1/10}+S_{N-1}^{(N)}\big)\,;\,
    \min_{k\le N-1}\Big(w+S_k^{(N)}+c_0h\big(k\wedge(N-1-k)\big)\Big)\ge0\bigg] \\
    \cdot\bP_{\lambda^*}^x\Big(y+y^{1/10}+S_1^{(N)}\in dw\Big).
\end{multline*}
Multiplying the above display by $N^{3/2}$ and letting $N\to\infty$, we see that the left hand side of \eqref{eq:ballot-Sk-limsup} is no more than
\[
\big(1+\overline\delta_{y\wedge z}\big)\int_0^\infty \bP_{\lambda^*}^x\Big(y+y^{1/10}+S_1\in dw\Big) \beta^*V(w)\int_{z+y^{1/10}}^{z+y^{1/10}+a}F(w'-y^{1/10})V(w')dw',
\]
where $\overline\delta$ is the analogue of $\delta$ when $h$ is replaced by $c_0h$, and this is in turn is bounded above by
\[
\big(1+\overline\delta_{y\wedge z}\big)\big(1+\epsilon_{y,z}\big)\beta^*V(y)\int_z^{z+a}F(w)V(w)dw,
\]
where
\begin{multline*}
\epsilon_{y,z}:=\sup_{w\in[z,z+a],x\in[0,1)}\bigg(\bE_{\lambda^*}^x\Big[V\big(y+y^{1/10}+S_1\big);y+y^{1/10}+S_1\ge-h(1)\Big]V\big(w+y^{1/10}\big)\\
-V(y)V(w)\bigg).
\end{multline*}
Since $V(y)/y\to1$ as $y\to\infty$, it follows that $\epsilon_{y,z}\to0$ as $y\wedge z\to\infty$, and thus \eqref{eq:ballot-Sk-limsup} follows.

Finally, we prove \eqref{eq:ballot-Sk-concave}. For $j=N$, this is a straightforward consequence of \eqref{eq:ballot-Sk-limsup} and \eqref{eq:ballot-Sk-limit}. For $j<N$, define $d_j^{(N)}:=|d_N|+h(N-j)/j$ and $S_k^{j,N}:=S_k+kd_j^{(N)}$. We recall from \cite[equation $(107)$]{BDZ} the inequality
\[
h\big(k\wedge(N-k)\big) \le 2h\big(k\wedge(j-k)\big) + (k/j)h(N-j),
\]
which implies
\begin{multline*}
    \bP_{\lambda^*}^x\left(y+S_j^{(N)} \in [z,z+a],\,\min_{k\le j}\Big(y+S_k^{(N)}+h\big(k\wedge(N-k)\big)\Big)\ge 0\right) \\
    \le \bP_{\lambda^*}^x\bigg(y+S_j^{(N,j)}\in[z+h(N-j),z+h(N-j)+a],\\
    \min_{k\le j}\Big(y+S_k^{(N,j)}+2h\big(k\wedge(j-k)\big)\Big)\ge0\bigg).
\end{multline*}
Since $N/2\le j<N$, one can show $0<d_j^{(N)} \le c_2\log j/j$ for a constant $c_2>0$ which depends on $c_0$ and $c_1$, but not on $j$ or $\{d_N\}$. Thus, \eqref{eq:ballot-Sk-concave} follows by applying the $j=N$ case to the above inequality. This completes the proof.
\qed

\subsection*{Acknowledgment} 
E.L.~was supported in part by NSF grant DMS-1812095.
This project has received funding from the European Research Council (ERC) under the European Union's Horizon 2020 research and innovation programme (grant agreement No.\ 692452), and from a US-Israel BSF grant.

\bibliographystyle{abbrv}
\bibliography{bbm}

\end{document}